\documentclass[11pt,a4paper,reqno]{amsart}
\usepackage{amsfonts,amssymb,amsmath,amsthm,graphicx,color,amscd,xspace,verbatim,mathrsfs}
\usepackage{scalerel}
\usepackage{esint}

\makeatletter

\usepackage{hyperref}

\makeatletter
\def\@tocline#1#2#3#4#5#6#7{\relax
  \ifnum #1>\c@tocdepth 
  \else
    \par \addpenalty\@secpenalty\addvspace{#2}%
    \begingroup \hyphenpenalty\@M
    \@ifempty{#4}{%
      \@tempdima\csname r@tocindent\number#1\endcsname\relax
    }{%
      \@tempdima#4\relax
    }%
    \parindent\z@ \leftskip#3\relax \advance\leftskip\@tempdima\relax
    \rightskip\@pnumwidth plus4em \parfillskip-\@pnumwidth
    #5\leavevmode\hskip-\@tempdima
      \ifcase #1
       \or\or \hskip 1em \or \hskip 2em \else \hskip 3em \fi%
      #6\nobreak\relax
      \dotfill
      \hbox to\@pnumwidth{\@tocpagenum{#7}}
    \par
    \nobreak
    \endgroup
  \fi}
\makeatother

\newtheorem{theorem}{Theorem}[section]
\newtheorem{lemma}[theorem]{Lemma}
\newtheorem{proposition}[theorem]{Proposition}

\newtheorem{corollary}[theorem]{Corollary}

\theoremstyle{definition}
\newtheorem{definition}[theorem]{Definition}

\theoremstyle{remark}
\newtheorem{remark}[theorem]{Remark}

\newcommand{\N}{{\mathbb N}}
\newcommand{\R}{{\mathbb R}}

\newcommand{\loc}{\mathrm{loc}}

\newcommand{\beqn}{\begin{eqnarray}}
\newcommand{\eeqn}{\end{eqnarray}}   
\newcommand{\beq}{\begin{eqnarray*}}
\newcommand{\eeq}{\end{eqnarray*}}

\newcommand{\be}{\begin{equation}}
\newcommand{\bel}[1]{\begin{equation}\label{#1}}
\newcommand{\ee}{\end{equation}}

\newcommand{\BA}{\begin{array}}
\newcommand{\EA}{\end{array}}
\newcommand{\BAN}{\renewcommand{\arraystretch}{1.2}
\setlength{\arraycolsep}{2pt}\begin{array}}
\newcommand{\BAV}[2]{\renewcommand{\arraystretch}{#1}
\setlength{\arraycolsep}{#2}\begin{array}}

\newcommand{\BSA}{\begin{subarray}}
\newcommand{\ESA}{\end{subarray}}
\newcommand{\BAL}{\begin{aligned}}
\newcommand{\EAL}{\end{aligned}}

\newcommand{\norm}[1]{\left \|#1\right \|}

\newcommand{\supp}{\mathrm{supp}\,}
\newcommand{\dist}{\mathrm{dist}\,}
\newcommand{\sign}{\mathrm{sign}}
\newcommand{\diam}{\mathrm{diam}\,}

\def\dist{\mathrm{dist}}

      \def\CL{{\mathcal L}}

   \def\BBN {\mathbb N}    
   \def\BBR {\mathbb R}

\def\GTM {\mathfrak M}

\newcommand{\xa}{\alpha}
\newcommand{\xb}{\beta}
\newcommand{\xg}{\gamma}

\newcommand{\xd}{\delta}
\newcommand{\xD}{\Delta}
\newcommand{\xe}{\varepsilon}
\newcommand{\xz}{\zeta}

\newcommand{\xk}{\kappa}

\newcommand{\xm}{\mu}
\newcommand{\xn}{\nu}

\newcommand{\xr}{\rho}
\newcommand{\xs}{\sigma}

\newcommand{\xf}{\phi}
\newcommand{\xF}{\Phi}

\newcommand{\xO}{\Omega}

\newcommand{\tp}{{\tau_+(s,\xm)}}
\newcommand{\tm}{{\tau_-(s,\xm)}}
\newcommand{\Hsg}{H_0^s(\Omega;|x|^\gamma)}

\newcommand{\Hsgz}{ H_0^s(\Omega \setminus \{0\};|x|^\gamma)}

\newcommand{\Hm}{\mathbf{H}_{\mu,0}^s(\Omega)}

\newcommand{\scI}{\mathscr{I}}

\newcommand{\1}{\textbf{1}}
\def\bal#1\eal{\small\begin{align*}#1\end{align*}\normalsize}
\def\ba#1\ea{\small\begin{align}#1\end{align}\normalsize}
\DeclareMathOperator*{\esssup}{ess\,sup}
\numberwithin{equation}{section}

\def\XXint#1#2#3{{\setbox0=\hbox{$#1{#2#3}{\int}$}
\vcenter{\hbox{$#2#3$}}\kern-.5\wd0}}

\begin{document}

\title[Fractional Hardy Poisson problem ]{Poisson  Problems involving fractional Hardy operators and measures}
\author[H. Chen]{Huyuan Chen}
\address{Huyuan Chen, Department of Mathematics, Jiangxi Normal University, Nanchang 330022, China}
\email{chenhuyuan@yeah.net}

\author[K. T. Gkikas]{Konstantinos T. Gkikas}
\address{K.T. Gkikas, Department of Mathematics, University of the Aegean, 
83200 Karlovassi, Samos, Greece\newline
Department of Mathematics, National and Kapodistrian University of Athens, 15784 Athens, Greece
}
\email{kgkikas@aegean.gr}

\author[P.T. Nguyen]{Phuoc-Tai Nguyen}
\address{Phuoc-Tai Nguyen, Department of Mathematics and Statistics, Masaryk University, Brno, Czech Republic}
\email{ptnguyen@math.muni.cz}


\maketitle

\begin{abstract}
In this paper, we study the Poisson problem involving a fractional Hardy operator and a measure source. The complex interplay between the nonlocal nature of the operator, the peculiar effect of the singular potential and the measure source induces several new fundamental difficulties in comparison with the local case. To overcome these difficulties, we perform a careful analysis of the dual operator in the weighted distributional sense and establish fine properties of the associated function spaces, which in turn allow us to formulate the Poisson problem in an appropriate framework. In light of the close connection between the Poisson problem and its dual problem, we are able to establish various aspects of the theory for the Poisson problem including the solvability, a priori estimates, variants of Kato's inequality and regularity results.
\end{abstract}

\bigskip

{\footnotesize \textit{Key words:  Poisson problem; Fractional Hardy Laplacian; Radon measure, Kato's inequality. }
	
	\smallskip
	
	 }

\tableofcontents

\section{Introduction}
\subsection{Overview of the literature}
The past decades have witnessed an increasing number of significant developments in the research of elliptic equations involving Hardy type operators due to their applications in various scientific disciplines. The effect of Hardy operators is elusive and  cannot be viewed simply as a lower order perturbation term of $(-\Delta)^s$. In this paper, we  devote special attention to the fractional Hardy operator of the form
\bal 
\CL_\mu^s  := (-\Delta)^s    +\frac{\mu}{|x|^{2s}},
\eal
which is constituted by two terms. The first one is the fractional Laplace operator $ (-\Delta)^s$, $s\in (0,1)$, defined by
\bal
(-\Delta)^s  u(x):= C_{N,s}\lim_{\epsilon\to 0^+} \int_{\R^N\setminus B_\epsilon(x) }\frac{ u(x)-
	u(y)}{|x-y|^{N+2s}}  dy.
\eal
Here $B_\epsilon(x)$ is the ball with center $x \in \R^N$ ($N \geq 2$) and radius $\epsilon>0$, and
\ba \label{CNs} C_{N,s}:=2^{2s}\pi^{-\frac N2}s\frac{\Gamma(\frac{N+2s}2)}{\Gamma(1-s)}>0
\ea
with $\Gamma$ being the Gamma function. The second term is the Hardy potential $\frac{\mu}{|x|^{2s}}$ which is singular at the origin. The value of the parameter $\mu \in \R$ has a profound influence on the analysis of $\CL_\mu^s$.

The Hardy operator $\CL_\mu^1:=-\Delta+ \mu |x|^{-2}$, which is the local version of $\CL_\mu^s$, 
appears in numerous contexts such as combustion models \cite{G}, quantum mechanic \cite{LL, KSWW-75} and control theory \cite{E,VZ}. The heat equation involving $\CL_\mu^1$ was first studied in \cite{VazZua-2000}. Sharp two-sided estimates for the heat kernel associated to $\CL_\mu^1$ was established in \cite{FilMosTer}. The effect of the Hardy potential on the existence and finite time blow-up solutions to Schr\"odinger equations was analyzed in \cite{Suz-16,MNN}. Singular solutions to semilinear elliptic equations with Hardy potentials have been studied in many papers; see, e.g.,  \cite{GV,C,ChVe,CQZ}. The topic on elliptic equations has been diversified in different directions, including \cite{FMT} concerning spectral properties of Hardy potentials with multipolar inverse-square potentials, \cite{ChVe1} for semilinear equations with Hardy potentials singular on the boundary, and \cite{D, GkiNg-linear, GkiNg-absorption, GkiNg-source} for equations involving more general potentials blowing up on a submanifold.

The investigation of the fractional Hardy operator $\CL_\mu^s$, $s\in(0,1)$, belongs to one of the hot topics in the area of PDE because of its wide-ranging interest to various fields in Mathematics and Physics. For instance, it is motivated by physical models related to
relativistic Schr\"odinger operator with Coulomb potential (see \cite{NRS,FLS1}) and by the study of Hardy inequalities and Hardy-Lieb-Thirring inequalities (see, e.g., \cite{FLS-2008,Fra-2009,tz}).  

The operator $\CL_\mu^s$ possesses intriguing properties. On the one hand, it bears analogous properties as the classical operator $\CL_\mu^1$. More precisely, $\CL_\mu^s$ is closely related to the fractional Hardy inequality
\ba \label{mu_00}
\frac{C_{N,s}}{2}\int_{\R^N}\int_{\R^N} \frac{|\varphi(x)-\varphi(y)|^2}{|x-y|^{N+2s}}dydx + \mu_0 \int_{\R^N} \frac{|\varphi(x)|^2}{|x|^{2s}}dx \geq 0, \quad \forall \varphi \in C_0^\infty(\R^N),
\ea
where the sharp constant in \eqref{mu_00} is explicitly determined by  (see, e.g., \cite{FLS-2008}) 
\bal
	\mu_0 =-2^{2s}\frac{\Gamma^2(\frac{N+2s}4)}{\Gamma^2(\frac{N-2s}{4})}.
\eal
Therefore when $\mu\geq \mu_0$,  $\CL^s_\mu$ is positive definite. Moreover, for $\mu \neq 0$, since the Hardy potential $\mu|x|^{-2s}$ is homogeneous with the same degree $-2s$ as $(-\Delta)^s$, it is is critical to the validity of the classical theory. On the other hand, unlike the local case, the nonlocality  of $(-\Delta)^s$ in interaction with the Hardy potential yields new types of difficulties in both methods employed and the computation level. 

Further properties of $\CL_\mu^s$ can be found in \cite{MY}. Sharp estimates for the heat kernel associated to $\CL_\mu^s$ were established in  \cite{B,JW}, which play an important role in the study of the corresponding Green function in the whole space $\R^N$ (see \cite{BBGM}). 

Recently,  it was shown in \cite[Proposition 1.2]{CW} that  for $\mu\geq \mu_0$, the equation
\bal \CL_\mu^s u=0\quad{\rm in}\ \ \R^N\setminus \{0\}
\eal
has two distinct radial solutions
\ba\label{fu} \Phi_{s,\mu}(x):=\left\{
\BAL
&|x|^{\tau_-(s,\mu)}\quad
	&&\text{if }  \mu>\mu_0\\
&|x|^{-\frac{N-2s}{2}}|\ln|x|| \quad  &&\text{if } \mu=\mu_0
\EAL
\right.\quad   \text{and}\ \ \Gamma_{s,\mu}(x):=|x|^{\tau_+(s,\mu)} \ \ \text{for } x \in \R^N \setminus \{0\},
\ea
where $\tau_-(s,\mu)  \leq  \tau_+(s,\mu)$.
The map $\mu\in[\mu_0,2s)\mapsto \tau_+(s,\mu)$ is continuous and increasing, while the map  $\mu\in[\mu_0,2s)\mapsto \tau_-(s,\mu)$ is continuous and decreasing. Moreover, 
\ba \label{tptm} \BAL
	&\tm+\tp =2s-N \quad  \text{for all }   \mu \geq \mu_0,\\
	&\tau_-(s,\mu_0)=\tau_+(s,\mu_0)=\frac{2s-N}2,\quad\
	\tau_-(s,0)=2s-N, \quad\ \tau_+(s,0)=0,  \\[1.5mm]
	&\lim_{\mu\to+\infty} \tm=-N\quad \text{ and}\quad \lim_{\mu\to+\infty} \tp=2s.
\EAL \ea

\textit{In the remaining of the paper, when there is no ambiguity, we write for short $\tau_+$ and $\tau_-$ instead of $\tau_+(s,\mu)$ and $\tau_-(s,\mu)$. }

It was also proved in \cite{CW} that
\bal
\int_{\R^N}\Phi_{s,\mu}   (-\Delta)^s_{\tau_+}\xi \, dx  =c_{s,\mu}\xi(0), \quad \forall\, \xi\in C^2_c(\R^N),
\eal
where $c_{s,\mu}>0$ and $(-\Delta)^s_{\tau_+}$ denotes the dual of the operator of $\CL^s_\mu$, which is  a weighted fractional Laplace operator  given by
\bal
(-\Delta)^s_{\tau_+} v(x):=
C_{N,s}\lim_{\epsilon\to0^+} \int_{\R^N\setminus B_\epsilon(x) }\frac{v(x)-
v(y)}{|x-y|^{N+2s}} \, |y|^{\tau_+}dy.
\eal
In addition, via the above weighted distributional form, isolated singularities for solutions of nonhomogeneous equation
\bal
\CL_\mu^s u= f \quad
{\rm in}\ \ \Omega\setminus\{0\}, \qquad
u\geq0 \quad   \text{ in }  \R^N\setminus  \Omega,
\eal
have been classified  under an optimal assumption for nonnegative function $f\in C^\beta_{\loc}(\bar \Omega\setminus\{0\})$ with $\beta\in(0,1)$.

For semilinear equations  with fractional Hardy potentials, we refer to \cite{F,M,WYZ,NRS}.
\medskip

\subsection{Introduction of the problem and main results}
Motivated by the above works, in the present paper, we aim to establish the existence, uniqueness  and qualitative properties of solutions to the Poisson problem involving the Hardy potential
\ba\label{eq 1.0}
\left\{ \BAL
\CL_\xm^s u &=\nu \quad &&\text{in }\;\xO,\\
u &= 0 \quad &&\text{in }\, \BBR^N\setminus\xO,
\EAL \right.
\ea
where  $\Omega\subset\R^N$ ($N\geq 2$) is a bounded open set containing the origin  and $\nu$ is a Radon measure on $\Omega$.

Problem \eqref{eq 1.0} has the following notable features. \smallskip

$\bullet$ To our knowledge, the existence of the Green function associated to $\CL_\mu^s$ in $\Omega$ has not been known in the literature, hence methods based on the Green representation cannot be applied to the study of \eqref{eq 1.0}. Our approach in this paper, inspired by \cite{CW}, is to analyze the associated weighted fractional Laplace operator $(-\Delta)^s_{\gamma}$ which is defined by
\bal
(-\Delta)^s_{\gamma} v(x):=C_{N,s}\lim_{\epsilon \to 0^+} \int_{\R^N\setminus B_\epsilon(x) }\frac{v(x)-
	v(y)}{|x-y|^{N+2s}} \, |y|^\gamma dy,
\eal
where  $\gamma\in\big[ -\frac{N-2s}{2}, 2s\big)$. Note that  $(-\Delta)^s_0$ reduces to the well-known fractional Laplace operator. From the integral-differential form of the weighted fractional Laplace operator, a natural restriction for the  function $v$ is
\bal \|v\|_{L_{2s-\gamma}(\R^N)}:=\int_{\R^N}\frac{|v(x)|}{(1+|x|)^{N+2s-\gamma}} dx<+\infty.
\eal
In light of the crucial link between $\CL_\mu^s$ and $(-\Delta)_{\gamma}^s$, the study of  problem \eqref{eq 1.0} is closely connected to the investigation of problem
\ba \label{eq:1-1}
\left\{
\BAL
(-\xD)^s_\gamma u &=f  \ \ && \text{in }  \Omega,  \\[0.5mm]
\qquad\quad u &=0  &&  \text{in } \R^N\setminus \Omega,
\EAL
\right.
\ea
where   $f:\xO\to\R$ is a measurable function. 

$\bullet$ The measure source $\nu$ requires to formulate the problem in an appropriate weak sense. Moreover, since the Hardy potential is singular at the origin, solutions to \eqref{eq 1.0} may exhibit a singularity profile near the origin, therefore we impose a condition regarding the behavior of test functions near the origin to guarantee the meaning of the weak formulation.

$\bullet$ The combined effect of the Hardy potential and the concentration of the source complicates the construction of solutions to problem \eqref{eq 1.0}. Therefore, for any given measure source on the whole domain $\Omega$, we will decompose it into two measures: a measure concentrated away from the origin and a Dirac measure concentrated at the origin. The case of Dirac source was treated in \cite{CW}, hence due to the linearity, it is sufficient to deal with measure source concentrated in $\Omega \setminus \{0\}$. \smallskip

Let us introduce the function spaces that we will work on in studying problem \eqref{eq 1.0} and problem \eqref{eq:1-1}. For $\xg\in[ \frac{2s-N}{2}, 2s),$ we denote by  $H_0^s(\xO;|x|^\gamma)$ the closure of the functions in $C^\infty(\R^N)$ with the compact support in $\Omega$  under the norm
\ba 
\norm{u}_{s,\gamma} :=\sqrt{\int_{\R^N}\int_{\R^N}\frac{|u(x)-u(y)|^2}{|x-y|^{N+2s}}|y|^\xg dy |x|^\xg dx}.
\ea
Note that $H_0^s(\xO;|x|^\gamma)$ is a Hilbert space with the inner product
\ba
\langle u,v\rangle_{s,\gamma}:=\int_{\R^N}\int_{\R^N}\frac{\big(u(x)-u(y)\big)\big(v(x)-v(y)\big)}{|x-y|^{N+2s}}|y|^\xg dy |x|^\xg dx.
\ea
For $\mu \geq \mu_0$, let  $\mathbf{H}^{s}_{\mu,0}(\xO)$ be the closure of the functions in $C^\infty(\R^N)$ with the compact support in $\Omega$ under the norm
\bal 
\norm{u}_{\mu}:=\sqrt{\frac{C_{N,s}}{2}\int_{\R^N}\int_{\R^N}\frac{|u(x)-u(y)|^2}{|x-y|^{N+2s}}  dy  dx+\mu\int_\Omega \frac{u(x)^2}{|x|^{2s}}dx}.
\eal
  This is a metric space with metric induced by the following quantity 
\ba \label{innermu}
\ll u,v\gg_{\mu}:=\frac{C_{N,s}}{2}\int_{\R^N}\int_{\R^N}\frac{\big(u(x)-u(y)\big)\big(v(x)-v(y)\big)}{|x-y|^{N+2s}} dy   dx+\mu  \int_\Omega \frac{u(x) v(x)}{|x|^{2s}}dx.
\ea
 When $\xm>\xm_0$,  $\mathbf{H}^{s}_{\mu,0}(\xO)$ is a Hilbert space with the inner product defined in \eqref{innermu}.   In the critical case $\xm=\xm_0,$ $\mathbf{H}^{s}_{\mu_0,0}(\xO)$ is no longer Hilbert space and we denote by the same notation $\mathbf{H}^{s}_{\mu_0,0}(\xO)$ its standard completion.  See Subsection \ref{subsec:Hm} for more details. 

Our first main result depicts important properties of, as well as deciphers the relation between, the above spaces.
\begin{theorem}\label{th:main-1}
 Assume that  $\xO$ is a  bounded Lipschitz domain containing the origin. 
	
	$(i)$ For any $\xg\in[ \frac{2s-N}{2}, 2s)$ and $\mu\geq\mu_0$,  the space  $C_0^\infty(\Omega\setminus\{0\})$ is dense in $\Hsg$ and in $\mathbf{H}_{\mu,0}^s(\Omega)$.
	 	
	$(ii)$ For any $\xg\in[ \frac{2s-N}{2}, 2s)$, there is $\mu\geq \mu_0$ such that $\tp=\gamma$ and
	\ba 
	\Hsg =\big\{|x|^{-\gamma}u: u \in \Hm \big\}.
	\ea
	
	$(iii)$ 
	  Let $\gamma\in[\frac{2s-N}{2},2s)$, $\beta<2s$ and $1\leq q<\min\Big\{  \frac{2N-2\beta}{N-2s},\ \frac{2N }{N-2s}\Big\}$. Then there exists a positive constant $c=c(N,\Omega,s,\gamma,\beta,q)$ such that
\ba 
	\big\| |\cdot|^{\gamma}v \big\|_{L^q(\Omega;|x|^{-\beta})} \leq c\, \| v \|_{s,\gamma}, \quad \forall v \in \Hsg.
\ea
\end{theorem}

The proof of statement (i) in Theorem \ref{th:main-1} is based on the choice of a special cut-off function and  some delicate estimates, which enable us to deal with the whole range $[\frac{2s-N}{2},2s)$. This result is tremendously useful in our analysis as, in many places, it allows us to work on smooth functions with compact support in $\Omega \setminus \{0\}$ instead of functions in $\Hsg$ or in $\Hm$; hence we are able to dwindle or to avoid serious issues coming from the singularity at $0$.   Statement (ii) shows the one-to-one correspondence between $\Hsg$ and $\Hm$ under the transformation $v=|x|^{-\gamma}u$ for $u \in \Hm$ and $v \in \Hsg$, which allows us to associate problem \eqref{eq 1.0} to problem \eqref{eq:1-1}. Statement (iii) is derived from Hardy inequalities and the equivalence between the norm in $\Hsg$ and the norm in $\Hm$. For related results on weighted fractional spaces, we refer the reader to \cite{DMPS,DV}.

We introduce the notion of variational solutions to \eqref{eq:1-1}.
\begin{definition}\label{def:varisol}
 Assume that  $\xg\in[ \frac{2s-N}{2}, 2s)$. A function $u$ is called a \textit{variational solution} to \eqref{eq:1-1}	if  $u \in \Hsg$ and
 \ba \label{varisol-form}
	\langle u,\xi\rangle_{s,\gamma} =(f,\xi)_\gamma \quad \forall \, \xi\in \Hsg,
 \ea
	where
\ba \label{innergamma}
(f,\xi)_\gamma:=\int_{\xO}f\xi |x|^\gamma dx.
\ea
\end{definition}

The next theorem gives the existence of a variational solution to problem \eqref{eq:1-1} and is obtained by using the standard variational method in combination with statement (iii) of Theorem \ref{th:main-1}. In addition, a Kato type inequality for the variational solution is also provided, which leads to the uniqueness result.

\begin{theorem}\label{th:main-2}
Let  $\xg\in[ \frac{2s-N}{2}, 2s)$, $\alpha\in\R$, $p>\max\Big\{  \frac{2N}{N+2s},\, \frac{2N+2\alpha}{N+2s},\ 1+\frac{\alpha}{2s}\Big\}$
 and $  f\in L^p(\Omega;|x|^{\alpha})$.
Then problem \eqref{eq:1-1} has a unique variational solution $u$. Moreover, there exists a constant $c=c(N,\Omega,s,\gamma,\alpha,p)$ such that 
\ba \label{varisol:est-1}
\norm{u}_{s,\gamma} \leq c\, \| f\|_{L^p(\Omega;|x|^{\alpha})}.
\ea
In addition, the following Kato type inequality holds
	\ba\label{kato1}
	\langle u^+,\xi \rangle_{s,\gamma}
	\leq (f\sign^+(u),\xi)_{\gamma}, \quad \forall \, 0\leq \xi\in \Hsg.
	\ea
\end{theorem}

We remark that if $\alpha<2s$ then the condition on $p$ in Theorem \ref{th:main-2} is reduced to \bal p>\max\left\{  \frac{2N}{N+2s},\, \frac{2N+2\alpha}{N+2s}\right\}.
\eal

Now we return to the study of problem \eqref{eq 1.0}. In order to give the definition of weak solutions, we introduce the space of test functions.


\begin{definition} \label{def:weaksol}
 Assume that  $\xO\subset\BBR^N$ is a bounded domain satisfying the exterior ball condition and containing the origin.  For $b<2s-\tau_+$, we denote by $\mathbf{X}_\xm(\xO;|x|^{-b})$ the space of functions $\psi$ with the following properties:
	
	(i) $\psi\in H^{s}_0(\xO;|x|^{\tau_+});$
	
	(ii) $(-\xD)^s_{\tau_+}\psi $ exists a.e. in $\xO\setminus\{0\}$ and $\displaystyle\sup_{x\in \xO\setminus\{0\}}\big||x|^b (-\xD)^s_{\tau_+}\psi(x)\big| <+\infty$;
	
	(iii) for any compact set $K\subset\xO\setminus\{0\}$, there exist $\xd_0>0$ and $w\in L^1_{\loc}(\xO\setminus \{0\})$ such that
	\bal
	\sup_{0<\xd\leq\xd_0}|(-\Delta)^s_{\tau_+,\xd}\psi|\leq w\;\; \text{a.e. in}\;K,
	\eal
where
	\bal
	(-\Delta)^s_{\tau_+,\delta} \psi(x):=
	C_{N,s} \int_{\R^N\setminus B_{\delta}(x) }\frac{\psi(x)-
		\psi(y)}{|x-y|^{N+2s}} \, |y|^{\tau_+}dy, \quad x \in \Omega \setminus \{0\}.
	\eal 
\end{definition}

The space $\mathbf{X}_\xm(\xO;|x|^{-b})$ of test functions  plays an essential role in the construction of distributional solutions and in the derivation of variants of Kato's inequality. When $s=1$, the solution of the problem
\bal
\left\{ \BAL 
\CL_\mu^* u &=1 &&\quad \text{in }  \Omega, \\
 u &=0 &&\quad \text{on }  \partial \Omega,
\EAL \right. 
\eal
where $\CL_\mu^*$ is the dual operator of $\CL^1_\mu$ in a weighted distributional sense \cite{CQZ}, could be used directly as a test function thanks to a careful analysis of its behavior near the origin via the ODE method. Nevertheless, in the nonlocal setting, this strategy fails. 

As it will be shown later by \eqref{estLinfty} and \eqref{smoothest1}, for any $\psi\in\mathbf{X}_\xm(\xO;|x|^{-b})$, there holds
 \ba\label{estLinfty3}
|\psi(x)|\leq Cd(x)^s\quad \text{for a.e. } x \in \xO,
 \ea
 where $d(x)={\rm dist}(x,\partial\Omega)$. 
Moreover, by Lemmas  \ref{estLinftylemma}--  \ref{smoothlemma2}  below, for $\alpha\in(0,s)$,
\ba\label{test sp}
C^2_0(\xO)\subset  \mathbf{X}_\xm(\xO;|x|^{-b}) \subset L^\infty(\Omega)\cap C^s_{\loc}(\bar \Omega\setminus\{0\})\cap C^{2s+\alpha}_{\loc}(\Omega\setminus\{0\}).
\ea

For $\alpha \in \R$, we denote by $\GTM(\Omega;d(x)^s|x|^{\alpha})$ (resp. $\GTM(\Omega\setminus\{0\};d(x)^s|x|^{\alpha})$) the space of Radon measures $\nu$ on $\Omega$ (resp. $\Omega\setminus\{0\}$) such that
\bal
\| \nu\|_{\GTM(\Omega;d(x)^s|x|^{\alpha})}:=\int_{\Omega}d(x)^s|x|^{\alpha} \, d|\nu|<+\infty, \\
\big(\text{resp. } \| \nu\|_{\GTM(\Omega\setminus\{0\};d(x)^s|x|^{\alpha})}:=\int_{\Omega\setminus\{0\}}d(x)^s|x|^{\alpha} \, d|\nu|<+\infty\big)
\eal
and by $\GTM^+(\Omega;d(x)^s|x|^{\alpha})$ (resp. $\GTM^+(\Omega\setminus\{0\};d(x)^s|x|^{\alpha})$) its positive cone.

 \begin{definition}
 Assume that  $\xO\subset\BBR^N$ is a bounded domain satisfying the exterior ball condition and containing the origin.  

$(i)$ Suppose $f\in L^1(\xO; d(x)^s|x|^{\tau_+})$. A function $u$ is called a \textit{weak solution} to problem
	\ba\label{veryweaksolution} \left\{ \BAL
	\CL_\xm^s u&=f&&\quad\text{in}\;\xO, \\
	u&=0&&\quad\text{in}\;\BBR^N\setminus\xO,
	\EAL \right.
	\ea
	if for any $b<2s-\tau_+$, $u\in L^1(\xO;|x|^{-b})$ and
\bal 
	\int_\xO u(-\xD)^s_{\tau_+}\psi dx=\int_\xO f\psi |x|^{\tau_+} dx, \quad\forall\psi\in \mathbf{X}_\xm(\xO;|x|^{-b}).
\eal

(ii) Suppose $\nu \in \GTM(\Omega \setminus \{0\},d(x)^s|x|^{\tau_+})$. A function $u$ is called a \textit{weak solution} to problem
\ba\label{veryweaksolution-nu} \left\{ \BAL
\CL_\xm^s u&=\nu&&\quad\text{in}\;\xO, \\
u&=0&&\quad\text{in}\;\BBR^N\setminus\xO,
\EAL \right.
\ea
if for any $b<2s-\tau_+$, $u\in L^1(\xO;|x|^{-b})$ and
\bal 
\int_\xO u(-\xD)^s_{\tau_+}\psi dx= \int_{\Omega \setminus \{0\}}\psi |x|^{\tau_+}d\nu, \quad\forall\psi\in \mathbf{X}_\xm(\xO;|x|^{-b}).
\eal
\end{definition}

%
%

The next result deals with the case of $L^1$ source. 
	
\begin{theorem}\label{existence2}
 Assume that  $\xO\subset\BBR^N$ is a bounded domain satisfying the exterior ball condition and containing the origin and $f\in L^1(\xO;d(x)^s|x|^{\tau_+})$. Then problem \eqref{veryweaksolution} admits a unique weak solution $u=u_f$. For any $b<2s-\tau_+$,	there exists a positive constant $c=c(N,\Omega,s,\mu,b)$ such that
	\ba \label{est:apriori-1}
	\| u \|_{L^1(\Omega;|x|^{-b})} \leq c\,\| f \|_{L^1(\Omega;d(x)^s|x|^{\tau_+})}.
	\ea
	Furthermore, there holds
	\ba \label{Kato:+-1}
	\int_\xO u^+(-\xD)^s_{\tau_+}\psi dx\leq\int_\xO f\sign^+(u)\psi |x|^{\tau_+} dx,\quad\forall \, 0\leq\psi\in \mathbf{X}_\xm(\xO;|x|^{-b})
	\ea
	and
	\ba \label{Kato||-1}
	\int_\xO |u|(-\xD)^s_{\tau_+}\psi dx\leq\int_\xO f\sign(u)\psi |x|^{\tau_+} dx,\quad\forall \, 0\leq\psi\in \mathbf{X}_\xm(\xO;|x|^{-b}).
	\ea
	As a consequence, the mapping $f \mapsto u$ is nondecreasing. In particular, if $f \geq 0$ then $u \geq 0$ a.e. in $\Omega \setminus \{0\}$.
\end{theorem}	

The variants of Kato's inequality \eqref{Kato:+-1} and \eqref{Kato||-1} are a main thrust of the present paper. The idea of the proof is to reduce problem \eqref{veryweaksolution} to the associated problem \eqref{eq:1-1} and then to make use of Kato type inequality \eqref{kato1}. Nevertheless, the derivation of \eqref{Kato:+-1} and \eqref{Kato||-1} does not follows straightforward from estimate \eqref{kato1}, but contains intermediate steps regarding an approximation procedure and a perturbation process. \smallskip

The existence and uniqueness result still holds when the source is a measure, as pointed out in the next result.

 \begin{theorem}\label{inter}
 Assume that $\xO\subset\BBR^N$ is a bounded domain satisfying the exterior ball condition and containing the origin and $\xn\in\mathfrak{M}(\Omega\setminus\{0\};d(x)^s|x|^{\tau_+})$. Then problem \eqref{veryweaksolution-nu} admits a unique weak solution $u=u_\nu$. For any $b<2s-\tau_+$, there exists a positive constant $c=c(N,\Omega,s,\mu,b)$ such that
	\ba \label{l1ineqb}
	\| u \|_{L^1(\Omega;|x|^{-b})} \leq c\,\| \nu \|_{\mathfrak{M}(\Omega\setminus\{0\};d(x)^s|x|^{\tau_+})}.
	\ea
	Moreover, the mapping $\nu \mapsto u$ is nondecreasing. In particular, if $\nu \geq 0$ then $u \geq 0$ a.e. in $\Omega \setminus \{0\}$.
\end{theorem} 

We note that the exterior ball condition can be relaxed if the weight $d(x)$ is not involved in the space of measure source. More precisely, if $\xO$ is an open bounded domain containing the origin, for any $\xn\in\mathfrak{M}(\Omega\setminus\{0\};|x|^{\tau_+})$, then Theorem \ref{existence2} and Theorem \ref{inter} still hold true.

When the source is a bounded measure in $\GTM(\Omega;d(x)^s |x|^{\tau_+})$,  it can be decomposed as the sum $\nu + \ell \delta_0$ where $\nu \in \GTM(\Omega\setminus \{0\};d(x)^s |x|^{\tau_+})$, $\ell \in \R$ and $\delta_0$ denotes the Dirac measure concentrated at the origin. 

Recall that the functions $\Phi_{s,\mu}$ and $\Gamma_{s,\mu}$, defined in \eqref{fu}, are solution of $\CL_\mu^s u=0$ in $\BBR^N\setminus\{0\}$.
By \cite[Theorem 4.14]{CW}, there exist a positive constant $c=c(N,s,\xm,\Omega)$ and a nonnegative function $\xF_{s,\xm}^\xO\in W^{s,2}_{\loc}(\BBR^N\setminus\{0\})$ such that $\xF_{s,\xm}^\xO=0$  in $\BBR^N\setminus\xO,$
\ba \label{PhiOmega}
\lim_{|x|\to 0}\frac{\xF_{s,\xm}^\xO(x)}{\xF_{s,\xm}(x)}=1\quad \text{and}\quad\xF_{s,\xm}^\xO(x)\leq c |x|^\tm\;\;\forall x\in \xO\setminus\{0\}.
\ea
Moreover
\bal
\ll \Phi_{s,\mu}^\Omega,\phi \gg_{\mu} = 0,\quad\forall \xf\in C_0^\infty(\xO\setminus\{0\})
\eal
and
\bal
\int_\xO \xF_{s,\xm}^\xO(-\xD)^s_{\tau_+}\psi dx=c_{s,\xm}\,\psi(0),\quad\forall\psi\in C^{1,1}_0(\xO),
\eal
where $c_{s,\xm}$ is a constant given in \cite[(1.15)]{CW}.

We note that  in \cite[Theorem 4.14]{CW} the fundamental solution   $\xF_{s,\xm}^\xO$ is constructed under the assumption that $\xO$ is $C^2$.  In fact, the $C^2$ smoothness of $\Omega$ can be relaxed and the assumption that  $\xO$ satisfies the exterior ball condition is sufficient for the existence of $\Phi_{s,\mu}^{\Omega}$ due to the
method of the super and sub solutions.

\begin{definition} Assume that   $\Omega$ is a bounded domain satisfying the exterior ball condition and containing the origin, $\xn\in\mathfrak{M}(\Omega\setminus\{0\};\, d(x)^s|x|^{\tau_+})$ and $\ell \in \R$.  We will say that $u$ is a \textit{weak solution} to
	\ba\label{veryweaksolutionmesure} \left\{ \BAL
	\CL_\xm^s u&=\xn + \ell \delta_0&&\quad\text{in}\;\xO\\
	u&=0&&\quad\text{in}\;\BBR^N\setminus\xO,
	\EAL \right.
	\ea
if  for any $b<2s-\tau_+$, $u\in L^1(\xO;|x|^{-b})$ and
 \ba \label{weakform-a1}
	\int_\xO u(-\xD)^s_{\tau_+}\psi dx=\int_{\xO\setminus\{0\}}\psi |x|^{\tau_+} d\xn + \ell \int_{\Omega} \Phi_{s,\mu}^{\Omega} (-\Delta)_{\tau_+}^s \psi dx, \ \ \forall\psi\in \mathbf{X}_\xm(\xO;|x|^{-b}).
	 \ea
\end{definition}

The following result states the solvability of problem \eqref{veryweaksolutionmesure}.

\begin{theorem}\label{existence3-dirac}
	 Assume that   $\Omega$ is a bounded domain satisfying the exterior ball condition and containing the origin, $\xn\in \mathfrak{M}(\Omega\setminus\{0\};d(x)^s|x|^{\tau_+})$ and $\ell \in \R$. Then problem \eqref{veryweaksolutionmesure} admits a unique weak solution $u = u_{\nu,\ell}$.  For any $b<2s-\tau_+$, there exists a positive constant $c=c(N,\Omega,s,\mu,b)$ such that
	\bal
	\| u \|_{L^1(\Omega;|x|^{-b})} \leq c(\| \nu \|_{\mathfrak{M}(\Omega\setminus\{0\};d(x)^s|x|^{\tau_+})}+\ell).
	\eal
	Moreover, the mapping $(\nu,\ell) \mapsto u$ is nondecreasing. In particular, if $\nu \geq 0$ and $\ell \geq 0$ then $u \geq 0$ a.e. in $\Omega \setminus \{0\}$.
\end{theorem}

 In a forthcoming article we will develop our observations to study qualitative properties of  the semilinear problem
\bal
 \left\{\BAL
\CL_\xm^s u+g(u)&= \nu  &&\quad\text{in}\;\xO\\
u&=0&&\quad\text{in}\;\BBR^N\setminus\xO,
\EAL \right. 
\eal
where  $\nu$ is a Radon measure defined in an appropriate framework and $g:\BBR\to\BBR$ is a nondecreasing continuous function such that $g(0)=0$. \medskip

\noindent \textbf{Organization of the paper.} The rest of this paper is organized as follows. In Section \ref{sec:funcset},  we prove density properties of $\Hsg$ and in $\mathbf{H}_{\mu,0}^s(\Omega)$ and  the relation between these spaces (Theorem \ref{th:main-1}). We also establish the related continuous and compact embeddings. Section \ref{sec:dual} is devoted to the study of variational solutions of
Poisson problems involving the dual operators (Theorem \ref{th:main-2}). In Section \ref{sec:nonhomogeneous}, we address the solvability for Poisson equations in $\Hsg$ and show local regularity results. The Poisson problem is treated in Section \ref{sec:Poisson}. In particular, we prove Theorem \ref{existence2} in Subsection \ref{subsec:L1source} and demonstrate Theorem \ref{inter} and Theorem \ref{existence3-dirac} in Subsection \ref{subsec:measuresource}.  \medskip

\noindent \textbf{Notation.} Throughout this paper, unless otherwise specified, we assume that $\Omega \subset \R^N$ ($N \geq 2)$ is a bounded domain containing the origin and $d(x)$ is the distance from $x \in \Omega$ to $\R^N \backslash \Omega$. We denote by $c, C, c_1, c_2, \ldots$ positive constants that may vary from one appearance to another and depend only on the data. The notation $c = c(a,b,\ldots)$ indicates the dependence of the constant $c$ on $a,b,\ldots$The constant $C_{N,s}$ is given by \eqref{CNs}. For a function $u$, we denote that $u^+=\max\{u,0\}$ and $u^- = \max\{-u,0\}$. For a set $A \subset \R^N$, the function $\1_A$ denotes the indicator function of $A$. \medskip

\noindent \textbf{Acknowledgements: }  {\small

H. Chen is supported by NSFC,  No. 12071189, 12361043,  by Jiangxi Province Science Fund No. 20232ACB201001.

K. T. Gkikas  is supported by the Hellenic Foundation for Research and Innovation (H.F.R.I.) under the “2nd Call for H.F.R.I. Research Projects to support Post-Doctoral Researchers” (Project
Number: 59). 

P.-T. Nguyen is supported by Czech Science Foundation, Project GA22-17403S.}\medskip

\section{Function setting} \label{sec:funcset}
In this section, we provide important properties of function spaces that we work on.
\subsection{Space $H_0^s(\Omega;|x|^\gamma)$}
We start with some estimates which are derived from Hardy inequalities.

For each $\xg\in[ \frac{2s-N}{2}, 2s)$, by \eqref{tptm}, there exists $\mu \geq \mu_0$ such that $\gamma=\tp$. For  $u \in C_0^\infty(\Omega)$,  put $v=|x|^{-\gamma}u \in C^\infty(\Omega \setminus \{0\})$, it can be checked that
\ba\label{normHs}
\frac{C_{N,s}}{2}\int_{\BBR^N}\int_{\BBR^N}\frac{|u(x)-u(y)|^2}{|x-y|^{N+2s}}dydx+\xm \int_{\BBR^N}\frac{u(x)^2}{|x|^{2s}}dx = \frac{C_{N,s}}{2}\int_{\BBR^N}\int_{\BBR^N}\frac{|v(x)-v(y)|^2}{|x-y|^{N+2s}}|y|^{\gamma}dy |x|^{\gamma}dx.
\ea

If $\mu>\mu_0$ then $\gamma>\frac{2s-N}{2}$. We infer from \eqref{mu_00} that
	\bal
	\frac{C_{N,s}}{2}\int_{\BBR^N}\int_{\BBR^N}\frac{|u(x)-u(y)|^2}{|x-y|^{N+2s}}dydx+\xm \int_{\BBR^N}\frac{u(x)^2}{|x|^{2s}}dx\geq (\mu-\mu_0) \int_{\BBR^N}\frac{u(x)^2}{|x|^{2s}}dx,
	\eal
which implies (note that $v$ has compact support in $\Omega$)

	\ba\label{subcrit1}\BAL
	\frac{C_{N,s}}{2}\int_{\BBR^N}\int_{\BBR^N}\frac{|v(x)-v(y)|^2}{|x-y|^{N+2s}}|y|^{\gamma}dy |x|^{\gamma}dx
	&\geq(\xm-\xm_0) \int_{\BBR^N}v(x)^2|x|^{2\gamma-2s}dx \\
	&\geq C(\Omega,s)(\xm-\xm_0) \int_{\BBR^N}v(x)^2|x|^{2\gamma}dx.
	\EAL
	\ea

If $\mu=\mu_0$ then $\gamma=\frac{2s-N}{2}$. Put $\displaystyle D_{\Omega}:=\sup_{x \in \Omega}|x|$ and denote
	\ba \label{X}
	X(t):=\left(1-\ln\frac{t}{D_{\Omega}}\right)^{-1} \quad \text{for } t>0.
	\ea
By \cite[Theorem 5]{tz},
	\ba \label{est:Hardy1}
	\frac{C_{N,s}}{2}\int_{\BBR^N}\int_{\BBR^N}\frac{|u(x)-u(y)|^2}{|x-y|^{N+2s}}dydx+\xm_0 \int_{\BBR^N}\frac{u(x)^2}{|x|^{2s}}dx\geq C(N,s) \int_{\BBR^N}\frac{u(x)^2X(|x|)^2}{|x|^{2s}}dx,
	\ea
which implies
	\ba\label{crit1}\BAL
	\frac{C_{N,s}}{2}\int_{\BBR^N}\int_{\BBR^N}\frac{|v(x)-v(y)|^2}{|x-y|^{N+2s}}|y|^{\frac{2s-N}{2}}dy |x|^{\frac{2s-N}{2}}dx
	&\geq C(N,s) \int_{\BBR^N}v(x)^2 X(|x|)^2|x|^{-N}dx \\
	&\geq C(N,\Omega,s) \int_{\BBR^N}v(x)^2 |x|^{2s-N}dx.
	\EAL
	\ea
From \eqref{subcrit1} and \eqref{crit1}, we may define the space $\Hsgz$ as the closure of the space of functions in $C^\infty(\R^N)$ with compact support in  $\Omega \setminus\{0\}$  under the norm
\ba \label{normsgamma}
\norm{u}_{s,\gamma}:=\left(\int_{\R^N}\int_{\R^N}\frac{|u(x)-u(y)|^2}{|x-y|^{N+2s}}|y|^\xg dy |x|^\xg dx \right)^{\frac{1}{2} }.
\ea
 Consequently,  $\Hsgz$  is a Hilbert space with inner product
\ba \label{innersgamma} \langle u,v\rangle_{s,\gamma}:=\int_{\R^N}\int_{\R^N}\frac{\big(u(x)-u(y)\big)\big(v(x)-v(y)\big)}{|x-y|^{N+2s}}|y|^\xg dy |x|^\xg dx. \ea
In the sequel, we will use the norm \eqref{normsgamma} and the inner product \eqref{innersgamma} in  $\Hsgz$.

\begin{proposition}\label{density}  Assume that  $\xg\in [\frac{2s-N}{2}, 2s)$. Then $C_0^\infty(\Omega\setminus\{0\})$ is dense in  $H^{s}_0(\xO;|x|^\xg)$. Consequently, $\Hsgz=H^{s}_0(\xO;|x|^\xg)$. 
\end{proposition}
\begin{proof}[\textbf{Proof}] The proof is split into two steps. \smallskip
	
\noindent \textbf{Step 1.}  We  show that
 $C_0^{0,1}(\Omega\setminus\{0\})$ is dense in  $H^{s}_0(\xO;|x|^\xg)$.
 For $u\in H^{s}_0(\xO;|x|^\xg)$ and $\xe>0$, we will show that there is a function $w \in C_0^{0,1}(\Omega \setminus \{0\})$ with support in $\Omega \setminus \{0\}$  such that
\ba \label{uw-1}
\| u-w \|_{s,\gamma}<\varepsilon.
\ea

Since $u\in H^{s}_0(\xO;|x|^\xg)$, by density, there exists $v\in C_0^\infty(\xO)$ such that
\ba\label{1a}
\norm{u-v}_{s,\gamma}<\frac{\xe}{3}.
\ea
For any $0<h<\min\{\frac{1}{4},\frac{\dist(0,\partial\xO)}{4}\}$, we consider the function
\ba \label{etah}
\eta_h(x):=\left\{ \BAL &1 \qquad&& \text{if }|x|>h, \\
&1-\frac{1}{\ln h}\ln\left(\frac{|x|}{h}\right)  &&\text{if } h^2\leq|x|\leq h,\\
&0&&\text{if }  |x|<h^2,
\EAL \right.
\ea
which is in $C^{0,1}(\R^N)$,
and set $v_h:=\eta_hv$. \medskip

Now we will show that
\ba \label{vvh}
\lim_{h\to 0}\norm{v-v_{h}}_{s,\gamma}=0.
\ea
Indeed, by the definition of the norm $\| \cdot\|_{s,\gamma}$, we have
\ba \label{vvh-2} \BAL
\norm{v-v_{h}}_{s,\gamma}^2 &=\int_{\{|x|\leq 2h\}}\int_{\{|y|\leq 2h\}}\frac{|v(x)-v_h(x)+v_h(y)-v(y)|^2}{|x-y|^{N+2s}}|y|^\xg dy |x|^\xg dx\\
&+2\int_{\{|x|> 2h\}}\int_{\{|y|\leq 2h\}}\frac{|v(x)-v_h(x)+v_h(y)-v(y)|^2}{|x-y|^{N+2s}}|y|^\xg dy |x|^\xg dx.
\EAL \ea

In the remaining of the proof,  $c$ denotes a generic constant which is dependent on $N,s$ and $\xg$, but independent of $h$ and $v$, and which may vary from one appearance to another.

The second term on the right-hand side of \eqref{vvh-2} is estimated as
\ba \label{vvh-3} \BAL
&\quad\ 2\int_{\{|x|> 2h\}}\int_{\{|y|\leq 2h\}}\frac{|v(x)-v_h(x)+v_h(y)-v(y)|^2}{|x-y|^{N+2s}}|y|^\xg dy |x|^\xg dx \\
&=2\int_{\{|x|> 2h\}}\int_{\{|y|\leq h\}}\frac{|v_h(y)-v(y)|^2}{|x-y|^{N+2s}}|y|^\xg dy |x|^\xg dx\\
&\leq c\,h^{\xg-2s} \int_{\{|y|\leq h\}}|v_h(y)-v(y)|^2|y|^\xg dy\\
&\leq c\, h^{\xg-2s}\norm{v}_{L^\infty(B(0,h))}^2\int_{\{|y|\leq h\}}|1-\eta_h(y)|^2|y|^\xg dy\\
&\leq c \norm{v}_{L^\infty(B(0,h))}^2 h^{N+2\xg-2s}  |\ln h|^{-2}
\\& \leq c\,  \norm{v}_{L^\infty(\Omega)}^2  h^{N+2\xg-2s}   |\ln h|^{-2},
\EAL \ea
where   $2\xg+N-2s\geq0$ and  $\norm{v}_{L^\infty(\Omega)}$   could be large, but it is independent of $h$.

The first term on the right-hand side of \eqref{vvh-2} is estimated as
\ba \label{vvhh-1} \BAL
&\quad\ \int_{\{|x|\leq 2h\}}\int_{\{|y|\leq 2h\}}\frac{|v(x)-v_h(x)+v_h(y)-v(y)|^2}{|x-y|^{N+2s}}|y|^\xg dy |x|^\xg dx \\
&\leq 2 \int_{\{|x|\leq 2h\}}\int_{\{|y|\leq 2h\}}\frac{(1-\eta_h(x))^2|v(x)-v(y)|^2}{|x-y|^{N+2s}}|y|^\xg dy |x|^\xg dx\\
&\quad+ 2 \int_{\{|x|\leq 2h\}}\int_{\{|y|\leq 2h\}}\frac{v(y)^2|\eta_h(x)-\eta_h(y)|^2}{|x-y|^{N+2s}}|y|^\xg dy |x|^\xg dx\\
&\leq 2 \int_{\{|x|\leq 2h\}}\int_{ \{|y|\leq 2h\}}\frac{ |v(x)-v(y)|^2}{|x-y|^{N+2s}}|y|^\xg dy |x|^\xg dx\\
&\quad+ 2\norm{v}_{L^\infty(B(0,2h))}^2 \int_{\{|x|\leq 2h\}}\int_{\{|y|\leq 2h\}}\frac{|\eta_h(x)-\eta_h(y)|^2}{|x-y|^{N+2s}}|y|^\xg dy |x|^\xg dx.
\EAL \ea
By the dominated convergence theorem we have
\ba \label{vvhh-2}
\lim_{h\to 0}\int_{\{|x|\leq 2h\}}\int_{ \{|y|\leq 2h\}}\frac{ |v(x)-v(y)|^2}{|x-y|^{N+2s}}|y|^\xg dy |x|^\xg dx=0.
\ea
Next we estimate the last term in \eqref{vvhh-1} by splitting it into several terms as follows
\ba\label{1} \BAL
&\quad\ \int_{\{|x|\leq 2h\}}\int_{\{|y|\leq 2h\}}\frac{|\eta_h(x)-\eta_h(y)|^2}{|x-y|^{N+2s}}|y|^\xg dy |x|^\xg dx \\
&=\int_{\{|x|<h^2\}}\int_{\{|y|<h^2\}} \ldots dx + \int_{\{|x|<h^2\}}\int_{\{h^2\leq|y|< 2h^2\}} \ldots dx \\
&\quad + \int_{\{|x|<h^2\}}\int_{\{2h^2 \leq|y|<h\}}\ldots dx + \int_{\{|x|<h^2\}}\int_{\{h \leq|y|\leq 2h\}} \ldots dx \\
&\quad +\int_{\{h^2\leq|x|<h\}}\int_{\{|y|<h^2\}} \ldots dx +\int_{\{h^2\leq|x|<h\}}\int_{\{h^2\leq|y|<h\}}\ldots dx 
\\ &\quad +\int_{\{h^2\leq|x|<h\}}\int_{\{h\leq|y|\leq 2h\}}\ldots dx + \int_{\{ h<|x|\leq 2h\}}\int_{\{|y| < h^2\}}\ldots dx \\
&\quad +\int_{\{ h \leq |x|\leq 2h\}}\int_{\{h^2\leq|y|<h\}}\ldots dx +\int_{\{ h \leq |x|\leq 2h\}}\int_{\{h\leq|y|\leq 2h\}} \ldots dx \\
&=:A_{h,1}+A_{h,2}+A_{h,3}+A_{h,4}+A_{h,5}+A_{h,6}+A_{h,7}+A_{h,8}+A_{h,9}+A_{h,10}.
\EAL \ea
We will estimate $A_{h,i}$, $i \in \{1,\ldots,10\}$, successively. \medskip

\noindent \textit{Estimate of $A_{h,1}$.} We note that $\eta_h(x)=0$ for any $|x|<h^2$, hence
\bal
A_{h,1}=0.
\eal

\noindent \textit{Estimate of $A_{h,2}$.} By the definition of $\eta_h$ in \eqref{etah}, the inequality $\ln t \leq t-1$ for any $t>0$ and the assumption that $\frac{2s-N}{2} \leq \gamma <2s$, we have
\ba \label{Ah2} \BAL
A_{h,2}&=|\ln h|^{-2}\int_{\{|x|<h^2\}}\int_{\{h^2\leq|y|<2h^2\}}\frac{|\ln\frac{|y|}{h^2}|^2}{|x-y|^{N+2s}}|y|^\xg dy |x|^\xg dx\\
&\leq h^{-4}|\ln h|^{-2}\int_{\{|x|<h^2\}}\int_{\{h^2\leq|y|<2h^2\}}\frac{||y|-h^2|^2}{|x-y|^{N+2s}}|y|^\xg dy |x|^\xg dx\\
&\leq h^{-4+2\xg}|\ln h|^{-2}\int_{\{|x|<h^2\}}\int_{\{h^2\leq|y|<2h^2\}}|x-y|^{-N-2s+2} dy |x|^\xg dx\\
&\leq  c\,  h^{2(N+2\gamma-2s)} |\ln h|^{-2}.
\EAL \ea
Here in the last estimate in \eqref{Ah2}, we have used the inequalities
\bal ||y|-h^2|^2\leq (|y|-|x|)^2\leq |y-x|^2 \quad{\rm for}\ \, |x|<h^2 \leq |y|. 
\eal

\noindent \textit{Estimate of $A_{h,3}$.} By the definition of $\eta_h$ in \eqref{etah} and since $|x-y|  \geq \frac{|y|}{2}$ for any $|x|<h^2$ and $2h^2 \leq  |y|<h$, we obtain
\ba \label{Ah3} \BAL
A_{h,3}&=|\ln h|^{-2} \int_{\{|x|<h^2\}}\int_{\{2h^2\leq|y|<h\}}\frac{|\ln\frac{|y|}{h^2}|^2}{|x-y|^{N+2s}}|y|^\xg dy |x|^\xg dx\\
&\leq c\, h^{2N+2\xg}|\ln h|^{-2}\int_{\{2h^2\leq|y|<h\}}|\ln\frac{|y|}{h^2}|^2|y|^{\xg-N-2s} dy \\
&\leq  c\, h^{2(N+2\gamma-2s)}|\ln h|^{-2}.
\EAL \ea

\noindent \textit{Estimate of $A_{h,4}$.} By the definition of $\eta_h$ in \eqref{etah}, we have
\bal
A_{h,4}=\int_{\{|x|<h^2\}}\int_{\{h\leq|y|\leq 2h\}}\frac{1}{|x-y|^{N+2s}}|y|^\xg dy |x|^\xg dx \leq c\, h^{-2s+\xg+2(N+\xg)}.
\eal

\noindent \textit{Estimate of $A_{h,5}$.}  By Fubini's theorem and by exchanging variables $x$ and $y$ in $A_{h,5}$, we obtain that 
\bal
&A_{h,5}=\int_{\{h^2\leq|x|<h\}}\int_{\{|y|<h^2\}}\frac{|\eta_h(x)-\eta_h(y)|^2}{|x-y|^{N+2s}}|y|^\xg dy |x|^{\xg} dx=A_{h,2} + A_{h,3} \leq c\, h^{2(N+2\gamma-2s)}|\ln h|^{-2}.
\eal
\noindent \textit{Estimate of $A_{h,6}$.} By the definition of $\eta_h$ in \eqref{etah}, we obtain
\bal
A_{h,6}&=|\ln h|^{-2}\int_{\{h^2\leq|x|<h\}}\int_{\{h^2\leq|y|<h\}}\frac{|\ln|x|-\ln|y||^2}{|x-y|^{N+2s}}|y|^\xg dy |x|^\xg dx\\
&\leq 2|\ln h|^{-2}\int_{\{h^2\leq|x|<h\}}\int_{\{|x|\leq|y|<h\}}\frac{|\ln|x|-\ln|y||^2}{|x-y|^{N+2s}} |y|^\xg dy |x|^{\xg} dx\\
&\leq 2|\ln h|^{-2}\int_{\{h^2\leq|x|<h\}}\int_{\{|x|\leq|y|<2|x|\}}\frac{|\ln|x|-\ln|y||^2}{|x-y|^{N+2s}} |y|^\xg dy |x|^{\xg} dx\\
&+ 2|\ln h|^{-2}\int_{\{h^2\leq|x|<h\}}\int_{\{2|x|\leq|y|<2h\}}\frac{|\ln|x|-\ln|y||^2}{|x-y|^{N+2s}} |y|^\xg dy |x|^{\xg} dx=:A_{h,6,1}+A_{h,6,2}.
\eal
Since $|\ln|x| - \ln|y|| \leq \frac{|x-y|}{|x|}$, we find
\bal
A_{h,6,1}&\leq 2|\ln h|^{-2}\int_{\{h^2\leq|x|<h\}}\int_{\{|x|\leq|y|<2|x|\}}|x-y|^{-N-2s+2}|y|^{\xg} dy |x|^{\xg-2} dx\\
&\leq c\,h^{2-2s}|\ln h|^{-2}\int_{\{h^2\leq|x|<h\}}|x|^{2\xg-2s}dx\\
&\leq \left\{
\BAL
&c\,h^{2-2s}|\ln h|^{-1}  &\quad\text{if}\;\;\xg=\frac{2s-N}{2},\\
&c\,h^{N+2\gamma+2-4s} |\ln h|^{-2}  &\quad\text{if}\;\;\xg>\frac{2s-N}{2}.
\EAL
\right.
\eal

Also,
\bal
A_{h,6,2}&\leq  2|\ln h|^{-2}\int_{\{h^2\leq|x|<h\}}\int_{\{2|x|\leq|y|<2h\}}\frac{|\ln|x|-\ln|y||^2}{|x-y|^{N+2s}} |y|^\xg dy |x|^{\xg} dx\\
&\leq c\,|\ln h|^{-2}\int_{\{h^2\leq|x|<h\}}\int_{\{2|x|\leq|y|<2h\}}|\ln\frac{|y|}{|x|}|^2|y|^{-N-2s+\xg} dy|x|^{\xg} dx\\
&\leq c\,|\ln h|^{-2}\int_{\{h^2\leq|x|<h\}}|x|^{2\xg-2s} dx\\
&\leq \left\{
\BAL
&c\, |\ln h|^{-1} &&\quad\text{if}\;\;\xg=\frac{2s-N}{2}\\
&c\, h^{N+2\xg-2s}|\ln h|^{-2} &&\quad\text{if}\;\;\gamma>\frac{2s-N}{2}.
\EAL
\right.
\eal
Combining the above estimates, we deduce
\bal
 A_{h,6}\leq c\,|\ln h|^{-1}.
\eal
\noindent \textit{Estimate of $A_{h,7}$.} By the definition of $\eta_h$ in \eqref{etah}, we have
\bal
A_{h,7}&=|\ln h|^{-2}\int_{\{h^2\leq|x|<h\}}\int_{\{h\leq|y|\leq 2h\}}\frac{|\ln\frac{|x|}{h}|^2}{|x-y|^{N+2s}}|y|^\xg dy |x|^\xg dx\\
&\leq |\ln h|^{-2}\int_{\{h^2\leq|x|<\frac{h}{2}\}}\int_{\{h\leq|y|\leq 2h\}}\frac{|\ln\frac{|x|}{h}|^2}{|x-y|^{N+2s}}|y|^\xg dy |x|^\xg dx\\
&\quad+|\ln h|^{-2}\int_{\{\frac{h}{2}\leq|x|\leq h\}}\int_{\{h\leq|y|\leq 2h\}}\frac{|\ln\frac{|x|}{h}|^2}{|x-y|^{N+2s}}|y|^\xg dy |x|^\xg dx \\
&=:A_{h,7,1}+A_{h,7,2}.
\eal
We have
\bal
A_{h,7,1}&\leq c\, h^{\xg-2s}|\ln h|^{-2}\int_{\{h^2\leq|x|<\frac{h}{2}\}}|\ln\frac{|x|}{h}|^2|x|^\xg dx \leq c\, h^{N+2\xg-2s}|\ln h|^{-2}.
\eal
By using the estimate $|\ln\frac{|x|}{h}| \leq 2\frac{|x-y|}{h}$ for $\frac{h}{2}\leq|x| \leq h \leq |y|$, we obtain
\bal
A_{h,7,2} \leq  c\, h^{N+2\gamma-2s}|\ln h|^{-2}.
\eal
From the above estimates, we derive
\bal
 A_{h,7}\leq c\, h^{N+2\gamma-2s}|\ln h|^{-2}.
\eal
\noindent \textit{Estimate of $A_{h,8}$.}  Using Fubini's theorem and exchanging variables $x$ and $y$ in $A_{h,8}$ lead to 
\bal
A_{h,8}=A_{h,4} \leq c\, h^{-2s+\xg+2(N+\xg)}.
\eal
\noindent \textit{Estimate of $A_{h,9}$.}  Similarly, using Fubini's theorem and exchanging variables $x$ and $y$ in $A_{h,9}$ imply 
\bal
A_{h,9}=A_{h,7}\leq c\,h^{N+2\gamma-2s}|\ln h|^{-2}.
\eal
\noindent \textit{Estimate of $A_{h,10}$.} By the definition of $\eta_h$ in \eqref{etah},
\bal
A_{h,10}=0.
\eal
Finally, by plugging the estimates of $A_{h,i}$, $i \in \{1,\ldots,10\}$ into \eqref{1}, we derive, for $h$ small,
\ba \label{vvhh-3}
 \int_{\{|x|\leq 2h\}}\int_{\{|y|\leq 2h\}}\frac{|\eta_h(x)-\eta_h(y)|^2}{|x-y|^{N+2s}}|y|^\xg dy |x|^\xg dx \leq c\,|\ln h|^{-1}.
\ea

Combining \eqref{vvh-2}--\eqref{vvhh-2} and \eqref{vvhh-3}, we deduce \eqref{vvh}.

As a consequence, there exists $h_0$ such that
\ba\label{2a}
\norm{v-v_{h_0}}_{s,\gamma}<\frac{\xe}{3}.
\ea
From \eqref{1a} and \eqref{2a}, we deduce
\bal
\norm{u-v_{h_0}}_{s,\gamma} < \frac{2\varepsilon}{3}.
\eal
Since $v_{h_0} \in C_0^{0,1}(\Omega \setminus \{0\})$, we obtain \eqref{uw-1} with $w=v_{h_0}$. Thus $C_0^{0,1}(\Omega\setminus\{0\})$ is dense in  $H^{s}_0(\xO;|x|^\xg)$.\medskip

\noindent \textbf{Step 2.} We show that $C_0^\infty(\Omega\setminus\{0\})$ is dense in  $H^{s}_0(\xO;|x|^\xg)$. For $u\in H^{s}_0(\xO;|x|^\xg)$ and $\xe>0$, we will show that there is a function $w \in C_0^{\infty}(\Omega \setminus \{0\})$ with support in $\Omega \setminus \{0\}$  such that
\ba \label{uw-11}
\| u-w \|_{s,\gamma}<\varepsilon.
\ea

Consider a sequence of mollifiers $\{\xz_n\}_{n \in \N}$. Let $n_0\in \BBN$ large enough such that
\bal
\overline{\supp(\xz_n*v_{h_0})}\cup\overline{\supp(v_{h_0})}\subset \xO_1\Subset\xO_2\Subset\xO\setminus\{0\},
\eal
where $\xO_1$ and $\xO_2$ are open sets. We will show that
\ba\label{2}
\lim_{n \to \infty}\norm{v_{h_0}-\xz_n*v_{h_0}}_{s,\gamma} = 0.
\ea
Indeed, we write
\bal
\norm{v_{h_0}-\xz_n*v_{h_0}}_{s,\gamma}^2
&=\int_{\xO_2}\int_{\xO_2}\frac{|v_{h_0}(x)-\xz_n \ast v_{h_0}(x)+\xz_n \ast v_{h_0}(y)-v_{h_0}(y)|^2}{|x-y|^{N+2s}}|y|^\xg dy |x|^\xg dx\\
&+2\int_{\BBR^N\setminus\xO_2}\int_{\xO_2}\frac{|v_{h_0}(x)-\xz_n \ast v_{h_0}(x)+\xz_n \ast v_{h_0}(y)-v_{h_0}(y)|^2}{|x-y|^{N+2s}}|y|^\xg dy |x|^\xg dx.
\eal
Since $\zeta_n \ast v_{h_0} \to v_{h_0}$ in $L^2(\Omega_1)$ as $n \to \infty$, we have
\bal
&\quad \int_{\BBR^N\setminus\xO_2}\int_{\xO_2}\frac{|v_{h_0}(x)-\xz_n*v_{h_0}(x)+\xz_n*v_{h_0}(y)-v_{h_0}(y)|^2}{|x-y|^{N+2s}}|y|^\xg dy |x|^\xg dx\\
&\leq c\,\int_{\xO_1}|\xz_n*v_{h_0}(y)-v_{h_0}(y)|^2dy\to 0 \quad \text{as } n \to \infty,
\eal
where $c>0$ depends on $N,s,\xg,\xO_1,\xO_2$.

We also find that
\bal
&\quad \int_{\xO_2}\int_{\xO_2}\frac{|v_{h_0}(x)-\xz_n*v_{h_0}(x)+\xz_n*v_{h_0}(y)-v_{h_0}(y)|^2}{|x-y|^{N+2s}}|y|^\xg dy |x|^\xg dx\\
&\leq c\, \int_{\xO_2}\int_{\xO_2}\frac{|v_{h_0}(x)-\xz_n*v_{h_0}(x)+\xz_n*v_{h_0}(y)-v_{h_0}(y)|^2}{|x-y|^{N+2s}}dy dx\to 0 \quad \text{as } n \to \infty.
\eal
From the above estimates, we get \eqref{2}. Consequently, there exists $n_0\in\BBN$ such that
\ba\label{3a}
\norm{v_{h_0}-\xz_{n_0}*v_{h_0}}_{s,\gamma}<\frac{\xe}{3}.
\ea
Finally combining \eqref{1a}, \eqref{2a} and \eqref{3a} yields
\bal
\| u - \zeta_{n_0} \ast v_{h_0} \|_{s,\gamma}  < \varepsilon.
\eal
Therefore, we obtain \eqref{uw-11} with $w=\zeta_{n_0} \ast v_{h_0}$. The proof is complete.
\end{proof}

\subsection{The space $\mathbf{H}_{\mu,0}^s(\Omega)$} \label{subsec:Hm}

  For any $\mu \geq \mu_0$, we deduce from \eqref{mu_00} that 
\ba \label{compare-norm}
\frac{C_{N,s}}{2}\left(1 + \frac{\mu^-}{\mu_0}\right) \| u \|_{s,0}^2\leq	 \norm{u}_{\mu}^2 \leq \frac{C_{N,s}}{2}\left(1 - \frac{\mu^+}{\mu_0}\right) \norm{u}_{s,0}^2, \quad \forall u\in C^\infty_0(\xO).
\ea
Therefore, for $\xm>\xm_0$,  the space $\mathbf{H}_{\mu,0}^s(\Omega)$ defined in the Introduction is a Hilbert space. 

Nevertheless, when $\xm=\xm_0$, the space $\mathbf{H}_{\mu_0,0}^s(\Omega)$ is no longer Hilbert. We point out below that $\mathbf{H}_{\mu_0,0}^s(\Omega)$  can be associated with $H_0^s(\Omega;|x|^{-\frac{N-2s}{2}})$.

Let  $\mathcal{Z}$ be the space of all Cauchy sequences $\{u_n\}$ in $\mathbf{H}_{\mu_0,0}^s(\Omega).$ We introduce the equivalence relation $\sim$ on $\mathcal{Z}$: for any $\{u_n\}, \{u_n'\} \in \mathcal{Z}$, we write 
$\{u_n\}\sim\{u_n'\}$ if and only if $\lim_{n\to\infty}\norm{u_n-u_n'}_{\xm_0}=0$. The quotient map associated with $\sim$ is the following surjective map
\bal
\mathcal{Q}: &\,\,\, \mathcal{Z}\;\;\; \to \mathbf{\tilde H}_{\mu_0,0}^s(\Omega):=\mathcal{Z}/\sim \\
&\{u_n\} \mapsto ~ [\{u_n\}].
\eal
Then the quotient space $\mathbf{\tilde H}_{\mu_0,0}^s(\Omega)$ endowed with the inner product
\bal
\xr([\{u_n\}],[\{ u_n'\}]):=\lim_{n\to\infty}\ll u_n, u_n'\gg_{\mu},
\eal
is a Hilbert space.

Let $i:\mathbf{H}_{\mu_0,0}^s(\Omega)\to \mathbf{\tilde H}_{\mu_0,0}^s(\Omega)$ be such that
\bal
i(u)=[\{u\}],\quad\forall x\in \mathbf{H}_{\mu_0,0}^s(\Omega),
\eal
where $\{u\}$ denotes the constant sequence in $\mathbf{H}_{\mu_0,0}^s(\Omega)$. Then $i(\mathbf{H}_{\mu_0,0}^s(\Omega))$ is dense in $\mathbf{\tilde H}_{\mu_0,0}^s(\Omega).$

We have the following observations.

(a) For each $\mathbf{u} \in \mathbf{\tilde H}_{\mu_0,0}^s(\Omega)$, by Proposition \ref{density2}, we can easily show that there exists a Cauchy sequence $\{u_n\}\subset C^\infty_0(\xO\setminus\{0\})$ such that $\mathbf{u}=[\{u_n\}].$ Set $\xg=-\frac{N-2s}{2}$ and $v_n=|x|^{-\xg}u_n.$ Then by \eqref{normHs}, we can easily see that $\{v_n\}$ is a Cauchy sequence in $H^{s}_0(\xO;|x|^\xg).$ This implies the existence of a function $v\in H^{s}_0(\xO;|x|^\xg)$ such that $v_n\to v$ in $H^{s}_0(\xO;|x|^\xg)$ and
\ba \label{rhouu}
\xr(\mathbf{u},\mathbf{u})= \lim_{n \to \infty} \| u_n \|_{\mu}^2 = \lim_{n \to \infty} \frac{C_{N,s}}{2} \| v_n \|_{s,\gamma}^2 =  \frac{C_{N,s}}{2}\norm{v}^2_{s,\xg}.
\ea

(b) Conversely, for any $v\in H^{s}_0(\xO;|x|^\xg)$, by Proposition \ref{density}, there exists a sequence $\{u_n\}\subset C^\infty_0(\xO\setminus\{0\})$ such that $v_n\to v$ in $H^{s}_0(\xO;|x|^\xg).$ Set $u_n=|x|^{\xg}v_n,$ then by \eqref{normHs}, $\{u_n\}\subset C_0^\infty(\xO\setminus\{0\})$ is a Cauchy sequence in $\mathbf{H}_{\mu_0,0}^s(\Omega).$ Hence there exists $\mathbf{u}\in\mathbf{\tilde H}_{\mu_0,0}^s(\Omega)$ such that $\mathbf{u}=[\{u_n\}]$ and \eqref{rhouu} holds.

In light of the above observations, we may identify the space $\mathbf{\tilde H}_{\mu_0,0}^s(\Omega)$ with the space
\bal
\mathcal{H}_0(\xO): =\{|x|^{\gamma}v: v \in H^{s}_0(\xO;|x|^\xg) \}
\eal
endowed with the norm
\bal
\norm{u}_{\mathcal{H}_0(\xO)}:=\sqrt{\frac{C_{N,s}}{2}}\norm{v}_{s,\xg}.
\eal

\textit{In the sequel, for the sake of convenience, we set $\mathbf{ H}_{\mu_0,0}^s(\Omega)=\mathcal{H}_0(\xO).$}
 
\begin{remark}
Let $\xg=-\frac{N-2s}{2}$ and $v\in C_0^\infty(\xO)$ such that $v=1$ in $B_\xe(0)$ for some $\xe>0$ small enough such that $B_{4\xe}(0)\subset\xO.$ Then $v\in  H^{s}_0(\xO;|x|^\xg)$, and hence $u=|x|^{\gamma}v\in \mathbf{ H}_{\mu_0,0}^s(\Omega).$ However,
\bal
\int_{\BBR^N}\int_{\BBR^N}\frac{|u(x)-u(y)|^2}{|x-y|^{N+2s}}dydx=-\xm_0 \int_{\BBR^N}\frac{u(x)^2}{|x|^{2s}}dx=+\infty.
\eal
\end{remark}

\begin{proposition} \label{density2}
For any $\mu \geq \mu_0$,  $C_0^\infty(\Omega\setminus\{0\})$ is dense in $\mathbf{H}_{\mu,0}^s(\Omega)$.	
\end{proposition}
\begin{proof}[\textbf{Proof}]
	For any $u\in \mathbf{H}_{\mu,0}^s(\Omega)$ and $\xe>0$, by the definition of $\mathbf{H}_{\mu,0}^s(\Omega)$, there exists $u_\xe\in C^\infty_0(\Omega)$  such that
	\ba \label{uuep-1} \|u- u_\xe\|_\mu < \frac \xe2.
	\ea
	Let
	 $v_\xe =|x|^{-\gamma}u_\xe \in C^\infty(\Omega \setminus \{0\})$  with $\gamma=\tau_+(s,\mu)$, then, by \eqref{normHs}, $v_\xe \in \Hsg$ and
	\bal \|u-|\cdot|^{\gamma}v_\xe\|_{\mu} < \frac \xe2.
	\eal
	From the proof of Proposition \ref{density},  there exists $\tilde v_\xe\in C_0^\infty(\Omega\setminus\{0\})$ such that
	\bal\sqrt{\frac{C_{N,s}}{2}}\|\tilde v_\xe-v_\xe\|_{s,\xg} < \frac \xe2.
	\eal
Put $\tilde u_{\varepsilon} =  |x|^{\gamma}\tilde v_\xe$ then $\tilde u_{\varepsilon} \in  C_0^\infty(\Omega\setminus\{0\})$ and from the equality \eqref{normHs}, we have
\ba \label{uuep-2}\| \tilde u_\xe - u_\xe\|_{\mu} < \frac \xe2.
\ea
Combining \eqref{uuep-1} and \eqref{uuep-2} implies
\bal \|\tilde u_\xe - u \|_{\mu} <  \xe. \eal
	By the arbitrariness of $\xe>0$, we derive that  $C_0^\infty(\Omega\setminus\{0\})$ is dense in  $\mathbf{H}_{\mu,0}^s(\Omega)$.
\end{proof}

Recall that the fractional Sobolev exponent is
\ba \label{Sob-exp}
2_s^*=\frac{2N}{N-2s}.
\ea

\begin{lemma}\label{lem_interpolation}
 Assume that  $\mu\geq \mu_0$, $\alpha < 2s$ and $1\leq q<\min\left\{2^*_s,\, \frac{2N-2\alpha}{N-2s}\right\}$. Then  $\Hm$ continuously and compactly embedded into $L^q(\Omega;|x|^{-\alpha})$.
Moreover, there exists a positive constant $c=c(N,\Omega,s,\mu,\alpha,q)$ such that
\ba\label{eq_interpolation}
	\| u \|_{L^q(\Omega;|x|^{-\alpha})}  \leq c\,\| u \|_{\mu}, \quad \forall \, u \in \Hm.
\ea

\end{lemma}
\begin{proof}[\textbf{Proof}] We consider two cases. \medskip

\noindent \textbf{Case 1:} $\alpha\leq 0$.	
If $\xm>\xm_0$, we infer from \eqref{compare-norm} that
$\mathbf {H}_{\mu,0}^s(\Omega) = H_{0}^s(\Omega)$.
	This and the well-known fractional embedding (see \cite{DPV}) imply that the embedding
	$\mathbf{H}_{\mu,0}^s(\Omega) \hookrightarrow L^q(\xO)
	$
	is continuous for any $q\in[1,2_s^*]$ and compact for any $q\in[1,2_s^*)$. If $\mu=\mu_0$ then $H_0^s(\xO)\varsubsetneq \mathbf {H}^{s}_{\mu_0, 0}(\xO)$
	and the embedding $\mathbf{H}^{s}_{\mu_0, 0}(\xO)\hookrightarrow L^q(\xO)$
	is continuous  and compact for any $q\in[1,2_s^*)$ (see \cite{Fra-2009}). Therefore, for any $q \in [1,2_s^*)$, there holds
	\ba \label{alpha=0}
	\| u \|_{L^q(\Omega;|x|^{-\alpha})}\leq c\,  \| u \|_{L^q(\Omega)} \leq c\,\| u \|_{\mu}, \quad \forall u \in \Hm,
	\ea
where $c=c(N,\Omega,s,\mu,\alpha,q)$. \medskip
	
\noindent \textbf{Case 2:}	$0 < \alpha < 2s$. Let $\tilde \alpha\in (\alpha,2s)$ be close enough to $2s$, then by using H\"older inequality and estimate \eqref{est:Hardy1}, we obtain
	 \ba \label{est:xasXa} \BAL
	 \int_{\Omega} |x|^{-\alpha }|u(x)|^q dx &= \int_{\Omega} |x|^{-\alpha}|u(x)|^{\frac{ 2 \alpha}{\tilde \alpha}} |u(x)|^{q-\frac{2 \alpha}{\tilde \alpha}} dx \\
	 &\leq \left( \int_{\Omega} |x|^{-\tilde \alpha}  |u(x)|^2 dx \right)^{\frac{\alpha}{\tilde \alpha}} \left( \int_{\Omega} |u(x)|^{ \frac{\tilde \alpha}{\tilde \alpha-\alpha}(q-\frac{2 \alpha}{\tilde \alpha})}dx \right)^{1-\frac{ \alpha}{\tilde \alpha} } \\
	 &\leq c\,\| u \|_{\mu}^q.
	 \EAL \ea
Here in the second estimate, we have used the inequality that $\frac{\tilde \alpha}{\tilde \alpha-\alpha}(q-\frac{2 \alpha}{\tilde \alpha}) < 2_s^*$ due to the choice that $\tilde \alpha$ is close enough to $2s$ and the assumption   $q<\frac{2N-2\alpha}{N-2s}$.
 Thus $\Hm$ is continuous embedded into $L^q(\Omega;|x|^{-\alpha})$ and (\ref{eq_interpolation}) follows by (\ref{est:xasXa}).

 Let
 $u_n \rightharpoonup 0$ weakly in $\Hm$ as $n\to+\infty$. Then $\{ u_n\}$ is uniformly bounded in $\Hm$ and by the compactness embedding $\Hm \hookrightarrow L^q(\Omega;|x|^{-\alpha})$,
 \bal \int_{\Omega} |u_n(x)|^{ \frac{\tilde \alpha}{\tilde \alpha-\alpha}(q-\frac{2 \alpha}{\tilde \alpha})}dx \to 0 
 \eal
 due to the inequality $\frac{\tilde \alpha}{\tilde \alpha-\alpha}(q-\frac{2 \alpha}{\tilde \alpha}) <2^*_s$ in \eqref{est:xasXa}. Consequently, by \eqref{est:xasXa}, we obtain
\bal
	 \int_{\Omega} |x|^{-\alpha }|u_n(x)|^q dx \to 0\quad{\rm as}\ n\to+\infty.
\eal
Therefore  $\Hm$ is compactly embedded in $L^q(\Omega;|x|^{-\alpha})$.
\end{proof}

\begin{proof}[\textbf{Proof of Theorem \ref{th:main-1}}.]  $(i)$ Statement $(i)$ follows from Proposition \ref{density} and Proposition \ref{density2}.

$(ii)$	 From the property of $\tau_+(s,\mu)$, for any $\xg\in[ \frac{2s-N}{2}, 2s)$, there exists $\mu\geq \mu_0$ such that
	$\tau_+(s,\mu)=\gamma$. Therefore, for any $u, v \in C_0^\infty (\xO\setminus\{0\})$ such that $u=|x|^{\gamma}v\in C_0^\infty (\xO\setminus\{0\})$, we have
	\ba \label{norm-equi}\norm{u}_{\mu}^2 = \frac{C_{N,s}}{2}\norm{v}_{s,\gamma}^2,
	\ea
	which, together with Proposition \ref{density} and Proposition \ref{density2}, implies that
	\bal   \Hsg =\{|x|^{-\gamma}u: u \in \Hm \}. \eal

$(iii)$ For any $v \in \Hsg$, by (ii), $u=|x|^{\gamma}v \in \Hm$ and 
\bal 
\int_{\Omega}|x|^{-\beta + \xg q} |v(x)|^q dx= \int_{\Omega}|x|^{-\beta } |u(x)|^q dx,
\eal
it follows from \eqref{eq_interpolation} that 
\bal
	\big\| |\cdot|^{\gamma} v \big\|_{L^q(\Omega;|x|^{-\beta })} \leq c\, \| v \|_{s,\gamma}, \quad \forall v \in \Hsg
\eal
for 
\bal q<\frac{2N-2 \beta}{N-2s}\quad{\rm if}\ \,  \beta\in(0,2s)\quad  {\rm and}\quad q<2^*_s\quad{\rm if}\ \,  \beta \leq 0,
\eal
where $c=c(N,\Omega,s,\gamma,\beta,q)$.
The proof is complete.
\end{proof}

When $\gamma=\frac{2s-N}{2}$, we show that the exponent $2_s^*$ is involved with a logarithmic correction. 

\begin{corollary}\label{coro_interpolation}
The embedding
\bal
H^{s}_0(\xO;|x|^{\frac{2s-N}{2}}) \hookrightarrow L^{ 2_s^*}(\Omega; |x|^{-N}X(|x|)^{\frac{2(N-s)}{N-2s}})\eal
is continuous, where the function $X$ is defined in \eqref{X}. Moreover, there exists a positive constant $c=c(N,s)$ such that
\ba \label{H-gamma-est-b}
	\|v\|_{L^{2_s^*}(\Omega; |x|^{-N}X(|x|)^{\frac{2(N-s)}{N-2s}})}\leq c\, \| v\|_{s,\frac{2s-N}{2}},  \quad \forall v \in H^{s}_0(\xO;|x|^{\frac{2s-N}{2}}).
\ea
\end{corollary}
\begin{proof}[\textbf{Proof}] Recall that  $\tau_+(s,\mu_0)=\frac{2s-N}{2}$. Take $v \in H^{s}_0(\xO;|x|^{\frac{2s-N}{2}})$ and put $u=|x|^{\frac{2s-N}{2}}v$.
Invoking \cite[Theorem 3]{tz}, we obtain
\bal
  \left(\int_{ \Omega}  X(|x|)^{\frac{2(N-s)}{N-2s} } |u(x)|^{2_s^*}dx\right)^{\frac{2}{2_s^*}} \leq c\, \| u \|_{\mu_0}^2,
\eal
where $c=c(N,s)$, which implies  \eqref{H-gamma-est-b}. The proof is complete.
\end{proof}

\section{Dual problems: variational solutions} \label{sec:dual}

We will show below the existence of a solution to the dual problem \eqref{eq:1-1} by using a variational method. The highlight of this section is a variant of Kato's inequality which implies the uniqueness of problem \eqref{eq:1-1}.  

\begin{proof}[\textbf{Proof of Theorem \ref{th:main-2}}] {\it Existence.} Consider the functional
\bal \scI(\varphi):= \frac{1}{2}\norm{\varphi}_{s,\gamma}^2-(f,\varphi)_\gamma, \quad\forall\, \varphi \in \Hsg.
\eal
We first see that $\scI$ is $C^1$ in $\Hsg$.

Since 
\ba \label{exs ex}
p>\max\left\{\frac{2N}{N+2s},\,  \frac{2N+2\alpha}{N+2s},\, 1+\frac{\alpha}{2s}\right\},
\ea
it follows that $p'<\min\big\{\frac{2N}{N-2s},\frac{2N-2\alpha\frac{p'}{p}}{N-2s}\big\}$ and $\alpha\frac{p'}{p}<2s$.
By using H\"older's inequality and Theorem \ref{th:main-1} (iii) with $q$ replaced by $p'$ and $\beta$ replaced by $\alpha\frac{p'}{p}$, we obtain that, for each $\varphi \in \Hsg$, 
\ba \label{varisol:est-2}
|(f,\varphi)_{\gamma}|   \leq   \| f \|_{L^{p}(\Omega;|x|^\alpha)} \big\| |\cdot|^\xg \varphi\big\|_{L^{p'}(\xO; |x|^{-\alpha\frac{p'}{p}})}
 \leq    c\, \| f \|_{L^{p}(\Omega;|x|^\alpha)} \norm{\varphi}_{s,\gamma},
\ea
where $c=c(N,\Omega,s,\gamma,\alpha,p)$.
This implies that $\scI$ is coercive on $\Hsg$.

Next we will show that $\scI$ is weakly lower semicontinuous on $\Hsg$. Let $\{\varphi_n\} \subset \Hsg$ such that  $\varphi_n \rightharpoonup \varphi$ weakly in $\Hsg$. By \eqref{varisol:est-2}, the linear operator $T$ defined by $T(\varphi):=(f,\varphi)_{\gamma}$ belongs to the dual of $\Hsg$. Hence, we have that
\bal
|T(\varphi_n -\varphi)|=|(f,\varphi_n -\varphi)_\gamma| \to 0 \quad \text{as } n \to \infty.
\eal
Next, we see that
\bal
\norm{\varphi_n}_{s,\gamma}^2 - \norm{\varphi}_{s,\gamma}^2 = \norm{\varphi_n-\varphi}_{s,\gamma}^2 + 2\langle \varphi, \varphi_n - \varphi \rangle_{s,\gamma},
\eal
which gives
\bal
\scI(\varphi_n) - \scI(\varphi) = \norm{\varphi_n-\varphi}_{s,\gamma}^2 + 2\langle \varphi, \varphi_n - \varphi \rangle_{s,\gamma} - (f,\varphi_n - \varphi)_{\gamma}.
\eal
This yields
\bal
\liminf_{n \to \infty}\scI(\varphi_n) \geq \scI(\varphi).
\eal		

Therefore $\scI$ has a critical point $u \in \Hsg$. It can be checked that $u$ is a variational solution of \eqref{eq:1-1}.
\medskip

\noindent \textit{Uniqueness.} The uniqueness follows from Kato type inequality \eqref{kato1}. \medskip

\noindent \textit{A priori estimate.} Estimate \eqref{varisol:est-1} can be obtained by taking $\xi=u$ in \eqref{varisol-form} and using estimate \eqref{varisol:est-2} with $\varphi=u$. \medskip

\noindent \textit{Kato type inequality.} The proof is in the spirit of \cite{KKP}.  Assume that  $u \in \Hsg$ is a variational solution of \eqref{eq:1-1}. Let $\xe>0$ and $0\leq\xz\in C_0^\infty(\xO\setminus \{0\})$. Put
\bal \eta_{\varepsilon}= \min\left\{1,\varepsilon^{-1}u^+ \right \} \quad \text{and} \quad \xf=\eta_{\xe}\zeta.
\eal
Note that $\xf \in \Hsg$, hence by taking $\xf$ as a test function in \eqref{varisol-form}, we have that
\ba \label{uetaep-1}
\int_{\BBR^N}\int_{\BBR^N}\frac{(u(x)-u(y))(\eta_{\xe}(x)\xz(x)-\eta_{\xe}(y)\xz(y))}{|x-y|^{N+2s}}|y|^{\xg}dy |x|^{\xg}dx
=\int_{\xO}f(x)\eta_{\xe}(x)\xz(x)|x|^{\xg}dx.
\ea
We will estimate the left-hand side of \eqref{uetaep-1} from below by dividing $\R^N \times \R^N$ into several subsets based on the value of $u$ in comparison with $0$ and $\varepsilon$.

First, since $\eta_{\varepsilon}(x)=0$ when $u(x) \leq 0$, it is easy to see that
\bal
\int_{\{u(x)\leq 0\}}\int_{\{u(y)\leq 0\}}\frac{(u(x)-u(y))(\eta_{\xe}(x)\xz(x)-\eta_{\xe}(y)\xz(y))}{|x-y|^{N+2s}}|y|^{\xg}dy |x|^{\xg}dx
=0.
\eal
Next, since $\eta_{\varepsilon}(x)=0$ when $u(x) \leq 0$ and $\eta_{\varepsilon}(y)=\varepsilon^{-1}u(y)$ when $0 <u(y)<\varepsilon$, we obtain
\ba \nonumber
&\quad \int_{\{u(x)\leq 0\}}\int_{\{0<u(y)<\xe\}}\frac{(u(x)-u(y))(\eta_{\xe}(x)\xz(x)-\eta_{\xe}(y)\xz(y))}{|x-y|^{N+2s}}|y|^{\xg}dy |x|^{\xg}dx\\ \label{aa1}
&=\xe^{-1}\int_{\{u(x)\leq 0\}}\int_{\{0<u(y)<\xe\}}\frac{(u(y)-u(x))u(y)\xz(y)}{|x-y|^{N+2s}}|y|^{\xg}dy |x|^{\xg}dx\geq0.
\ea
We note that $\eta_{\varepsilon}(x)=0$ when $u(x) \leq 0$ and $\eta_{\varepsilon}(y)=1$ when $u(y) \geq \varepsilon$, hence
\ba \nonumber
&\quad \int_{\{u(x)\leq 0\}}\int_{\{u(y)\geq\xe\}}\frac{(u(x)-u(y))(\eta_{\xe}(x)\xz(x)-\eta_{\xe}(y)\xz(y))}{|x-y|^{N+2s}}|y|^{\xg}dy |x|^{\xg}dx\\ \nonumber
&=\int_{\{u(x)\leq 0\}}\int_{\{u(y)\geq\xe\}}\frac{(u(y)-u(x))\xz(y)}{|x-y|^{N+2s}}|y|^{\xg}dy |x|^{\xg}dx\\ \label{aa2}
&\geq \int_{\{u(x)\leq 0\}}\int_{\{u(y)\geq\xe\}}\frac{(u^+(x)-u^+(y))(\xz(x)-\xz(y))}{|x-y|^{N+2s}}|y|^{\xg}dy |x|^{\xg}dx.
\ea
By symmetry, as in \eqref{aa1}, we have
\bal
\int_{\{0<u(x)<\xe\}}\int_{\{u(y)\leq0\}}\frac{(u(x)-u(y))(\eta_{\xe}(x)\xz(x)-\eta_{\xe}(y)\xz(y))}{|x-y|^{N+2s}}|y|^{\xg}dy |x|^{\xg}dx\geq 0.
\eal
Again, by using the fact that $\eta_{\varepsilon}(x)=\varepsilon^{-1}u(x)$ when $0 <u(x)<\varepsilon$, we derive
\bal
&\quad \int_{\{0<u(x)<\xe\}}\int_{\{0<u(y)<\xe\}}\frac{(u(x)-u(y))(\eta_{\xe}(x)\xz(x)-\eta_{\xe}(y)\xz(y))}{|x-y|^{N+2s}}|y|^{\xg}dy |x|^{\xg}dx\\
&=\xe^{-1}\int_{\{0<u(x)<\xe\}}\int_{\{0<u(y)<\xe\}}\frac{(u(x)-u(y))(u(x)\xz(x)-u(y)\xz(y))}{|x-y|^{N+2s}}|y|^{\xg}dy |x|^{\xg}dx\\
&=\xe^{-1}\int_{\{0<u(x)<\xe\}}\int_{\{0<u(y)<\xe\}}\frac{(u(x)-u(y))^2\xz(x)}{|x-y|^{N+2s}}|y|^{\xg}dy |x|^{\xg}dx\\
&\quad +\xe^{-1}\int_{\{0<u(x)<\xe\}}\int_{\{0<u(y)<\xe\}}\frac{(u(x)-u(y))u(y)(\xz(x)-\xz(y))}{|x-y|^{N+2s}}|y|^{\xg}dy |x|^{\xg}dx\\
&\geq \xe^{-1}\int_{\{0<u(x)<\xe\}}\int_{\{0<u(y)<\xe\}}\frac{(u(x)-u(y))u(y)(\xz(x)-\xz(y))}{|x-y|^{N+2s}}|y|^{\xg}dy |x|^{\xg}dx.
\eal
Since $u \in H_0^s(\Omega;|x|^\gamma)$ and $\zeta \in C_0^\infty(\Omega \setminus \{0\})$, by using the dominated convergence theorem, we deduce that
\bal
\lim_{\xe\to0}\xe^{-1}\int_{\{0<u(x)<\xe\}}\int_{\{0<u(y)<\xe\}}\frac{(u(x)-u(y))u(y)(\xz(x)-\xz(y))}{|x-y|^{N+2s}}|y|^{\xg}dy |x|^{\xg}dx=0.
\eal
Next, as $\eta_{\varepsilon}(x)=\varepsilon^{-1}u(x)$ when $0 <u(x)<\varepsilon$ and $\eta_{\varepsilon}(y)=1$ when $u(y) \geq \varepsilon$, it follows that
\ba \nonumber
&\quad \int_{\{0<u(x)<\xe\}}\int_{\{ u(y) \geq \xe\}}\frac{(u(x)-u(y))(\eta_{\xe}(x)\xz(x)-\eta_{\xe}(y)\xz(y))}{|x-y|^{N+2s}}|y|^{\xg}dy |x|^{\xg}dx\\ \nonumber
&=\varepsilon^{-1}\int_{\{0<u(x)<\xe\}}\int_{\{ u(y) \geq \xe\}}\frac{(u(x)-u(y))(u(x)\xz(x)-\varepsilon\xz(y))}{|x-y|^{N+2s}}|y|^{\xg}dy |x|^{\xg}dx\\ \nonumber
&=\varepsilon^{-1}\int_{\{0<u(x)<\xe\}}\int_{\{ u(y) \geq \xe\}}\frac{(u(x)-u(y))(\xz(x)-\xz(y))u(x)}{|x-y|^{N+2s}}|y|^{\xg}dy |x|^{\xg}dx\\ \nonumber
&\quad   +\varepsilon^{-1}\int_{\{0<u(x)<\xe\}}\int_{\{ u(y) \geq \xe\}}\frac{(u(x)-u(y))(u(x)-\xe)\xz(y)}{|x-y|^{N+2s}}|y|^{\xg}dy |x|^{\xg}dx\\ \label{aa3}
&\geq \varepsilon^{-1}\int_{\{0<u(x)<\xe\}}\int_{\{ u(y) \geq \xe\}}\frac{(u(x)-u(y))(\xz(x)-\xz(y))u(x)}{|x-y|^{N+2s}}|y|^{\xg}dy |x|^{\xg}dx.
\ea
Noting that $u \in H_0^s(\Omega;|x|^\gamma)$ and $\zeta \in C_0^\infty(\Omega \setminus \{0\})$, by using the dominated convergence theorem, we obtain
\bal
\lim_{\xe\to0}\varepsilon^{-1}\int_{\{0<u(x)<\xe\}}\int_{\{ u(y) \geq \xe\}}\frac{(u(x)-u(y))(\xz(x)-\xz(y))u(x)}{|x-y|^{N+2s}}|y|^{\xg}dy |x|^{\xg}dx=0.
\eal
By symmetry, proceeding as in \eqref{aa2}, we get
\bal
&\quad \int_{\{u(x)\geq\xe\}}\int_{\{u(y)\leq0\}}\frac{(u(x)-u(y))(\eta_{\xe}(x)\xz(x)-\eta_{\xe}(y)\xz(y))}{|x-y|^{N+2s}}|y|^{\xg}dy |x|^{\xg}dx\\
&\geq \int_{\{u(x)\geq\xe\}}\int_{\{u(y)\leq0\}}\frac{(u^+(x)-u^+(y))(\xz(x)-\xz(y))}{|x-y|^{N+2s}}|y|^{\xg}dy |x|^{\xg}dx
\eal
and, as in \eqref{aa3},
\ba \nonumber
&\quad \int_{\{u(x)\geq\xe\}}\int_{\{0<u(y)<\xe\}}\frac{(u(x)-u(y))(\eta_{\xe}(x)\xz(x)-\eta_{\xe}(y)\xz(y))}{|x-y|^{N+2s}}|y|^{\xg}dy |x|^{\xg}dx \\ \label{aa4}
&\geq \varepsilon^{-1} \int_{\{u(x)\geq\xe\}}\int_{\{0<u(y)<\xe\}}\frac{(u(x)-u(y))(\xz(x)-\xz(y))u(y)}{|x-y|^{N+2s}}|y|^{\xg}dy |x|^{\xg}dx,
\ea
where
\bal
\lim_{\xe\to0}\varepsilon^{-1} \int_{\{u(x)\geq\xe\}}\int_{\{0<u(y)<\xe\}}\frac{(u(x)-u(y))(\xz(x)-\xz(y))u(y)}{|x-y|^{N+2s}}|y|^{\xg}dy |x|^{\xg}dx=0.
\eal

Combining all the above estimates and then letting $\xe\to 0$, we deduce that
\ba \label{kato1-1}
\int_{\BBR^N}\int_{\BBR^N}\frac{(u^+(x)-u^+(y))(\xz(x)-\xz(y))}{|x-y|^{N+2s}}|y|^{\xg}dy |x|^{\xg}dx
\leq\int_{\xO}f(x)\sign^+(u(x))\xz(x)|x|^{\xg}dx
\ea
for any $0\leq \xz\in C^\infty_0(\xO\setminus \{0\})$. By Propostion \ref{density}, we conclude that \eqref{kato1-1} holds true for any $0 \leq \zeta \in H_0^s(\Omega;|x|^{\gamma})$.
The proof is complete.
\end{proof}\medskip

\section{Nonhomogeous linear equations} \label{sec:nonhomogeneous}
\subsection{Existence, uniqueness and a priori estimates} ~  Assume that   $N\geq2$ and $\xO\subset\BBR^N$ is an open set. We denote by $H^s(\xO)$ the Banach space
\bal
H^s(\xO):=\left\{u\in L^2(\xO):\;\int_{\xO}\int_{\xO}\frac{|u(x)-u(y)|^2}{|x-y|^{N+2s}}dydx<\infty\right\}
\eal
endowed with the norm
\bal
\norm{u}_{H^s(\xO)}^2:=\int_{\xO}|u|^2dx+\int_{\xO}\int_{\xO}\frac{|u(x)-u(y)|^2}{|x-y|^{N+2s}}dydx.
\eal

We say that $u\in H^s_{\loc}(\xO)$ if  $u\in H^s(\xO')$ for any domain $\xO' \Subset\xO$.

In the following, unless otherwise stated, we assume that $\xO\subset\BBR^N$ is an open set containing the origin.

\begin{definition}\label{weaksolution}
Let $f \in L_{\loc}^1(\Omega \setminus \{0\})$. A function $u$  is called a \textit{weak solution} of
\ba \label{solutionL2}
\begin{aligned}
\CL_\xm^s u = f\quad \text{in}\;\xO
\end{aligned}
\ea
if $u\in H^s_{\loc}(\xO\setminus\{0\})\cap L^1(\BBR^N;(1+|x|^{N+2s})^{-1})$ and 
\ba \label{def:weaksol0}
\ll u, \phi \gg_{\mu} = \int_{\Omega}f \phi \, dx, \quad \forall \phi \in C_0^\infty(\Omega \setminus \{0\}).
\ea
\end{definition}

The solvability and a priori estimate for solutions of \eqref{solutionL2} in $\Hm$ is provided in the next result.

\begin{proposition}\label{existence}
Let $f\in L^2(\xO;|x|^a)$ for some $a <2s$. Then there exists a unique weak solution $u$ of \eqref{solutionL2} such that $u \in \Hm$. Furthermore, there exists a positive constant $c=c(N,\xO,s,\xm,a)$ such that
\ba\label{estL2}
\| u \|_{\mu} \leq c\,\| f \|_{L^2(\Omega;|x|^a)}.
\ea
\end{proposition}
\begin{proof}[\textbf{Proof}]
\noindent \textit{Existence.} Note that $\tau_+ \in [\frac{2s-N}{2},2s)$. Since $f \in L^2(\Omega; |x|^a)$ with $a<2s$, by Theorem \ref{th:main-2}, there exists a unique variational $v \in H_0^s(\Omega;|x|^{\tau_+})$ of the equation
\ba \label{eq:v-zeta}
\langle v, \zeta \rangle_{s,\tau_+} = (f,\zeta)_{\tau_+}, \quad \forall \zeta \in H_0^s(\Omega;|x|^{\tau_+}).
\ea
Moreover,
\ba \label{est-vf}
\| v \|_{s,\tau_+} \leq c \,\| f\|_{L^2(\Omega;|x|^a)}.
\ea

Put $u:=|x|^{\tau_+}v$ then $u \in \Hm$ due to Theorem \ref{th:main-1} (ii). Let $\phi \in C_0^\infty(\Omega \setminus \{0\})$ and put $\xz:=|x|^{-\tau_+}\xf$. Since $v,\zeta$ satisfy \eqref{eq:v-zeta}, it follows that $u,\phi$ satisfy \eqref{def:weaksol0}, hence $u$ is a weak solution of \eqref{solutionL2}. \medskip

\noindent \textit{Uniqueness.}
Let $\xf\in C_0^\infty(\xO\setminus\{0\})$.   Assume that  $u$ is a weak solution of \eqref{solutionL2} such that $u \in \Hm$. Put $v=|x|^{-\tau_+}u$ then $v \in \Hsg$ by Theorem \ref{th:main-1} (ii). Take $\zeta \in C_0^\infty(\Omega \setminus \{0\})$ and put $\phi=|x|^{\tau_+}\zeta \in C_0^\infty(\Omega \setminus \{0\})$. Since $u,\phi$ satisfy \eqref{def:weaksol0}, it follows that $v,\zeta$ satisfy \eqref{eq:v-zeta}. In light of Theorem \ref{density}, we deduce \eqref{eq:v-zeta}. Since $v$ is the unique variational solution of \eqref{eq:v-zeta}, we find that $u$ is the unique weak solution of \eqref{solutionL2}. \medskip

\noindent \textit{A priori estimate.} Estimate \eqref{estL2} follows from \eqref{est-vf} and \eqref{norm-equi}. 
\end{proof}

Denote
\bal
L^\infty(\xO;|x|^b):=\{u\in L^\infty_{\mathrm{loc}}(\xO): |x|^b u\in L^\infty(\xO)\},
\eal
with the norm
\bal
\norm{u}_{L^\infty(\xO;|x|^b)}=\esssup_{x\in\xO} (|u(x)||x|^b).
\eal

\begin{remark} \label{LinftyL2} We note that, for any $b<2s-\tau_+$, there exists $0\leq a=a(N,b,s,\mu)<2s$ such that $L^\infty(\xO;|x|^b) \subset L^2(\xO;|x|^a)$. Indeed, since
$2b-N < 2s$, we can choose $a=a(N,b,s,\mu) \in [0,2s)$ such that $a>2b-N$. Then for any $f \in L^\infty(\xO;|x|^b)$, we have
\bal
\int_{\Omega} |f(x)|^2 |x|^a dx \leq \| f\|_{L^\infty(\Omega;|x|^b)}^2 \int_{\Omega}|x|^{a-2b}dx =c\, \| f\|_{L^\infty(\Omega;|x|^b)}^2,
\eal
therefore $L^\infty(\xO;|x|^b) \subset L^2(\xO;|x|^a)$. In case $b<\frac{N}{2}$, we can choose $a\leq 0$, hence $L^\infty(\Omega;|x|^a) \subset L^2(\Omega)$.
\end{remark}
\begin{lemma}\label{estLinftylemma}
Let $b<2s-\tp$ and $f\in L^\infty(\xO;|x|^b)$.  Assume that  $u$  is a weak solution of \eqref{solutionL2} such that $u \in \Hm$. Then there exists a positive constant $c=c(N,\xO,s,\xm,b)$ such that
\ba\label{estLinfty}
|u(x)|\leq c\norm{f}_{L^\infty(\xO;|x|^b)}|x|^{\tp}\quad \text{for a.e. } x \in \xO\setminus\{0\}.
\ea
\end{lemma}
\begin{proof}[\textbf{Proof}] In this proof, we will modify $\mu$, hence we will use the notation $\tp$ to avoid confusion.
	
Since $u$  is a solution of \eqref{solutionL2} such that $u\in \Hm$, by putting $v=|x|^{-\tp}u$ then $v \in H^{s}_0(\xO;|x|^{\tp})$ and, as in the proof of the uniqueness part of Lemma \ref{existence}, it satisfies
\ba \label{veq-a}
\langle v,\zeta \rangle_{s,\tp} = (f,\zeta)_{\tp}, \quad \forall \xz\in H^{s}_0(\xO;|x|^{\tp}).
\ea

Next, we construct an upper bound for $v$, which will yield an upper bound for $u$. Let $\varepsilon>0$ small which will be determined later. By \cite[Proposition 1.2]{CW}, we have
\ba \label{xepsilon}
\CL_\mu^s(|x|^{\tau_+(s,\xm+\xe)})=-\xe|x|^{\tau_+(s,\xm+\xe)-2s}\leq 0\quad\text{in}\;\xO  \setminus \{0\}.
\ea

Let $\eta \in C^\infty(\R^N)$ such that $0\leq\eta\leq1$, $\eta(x)=1$ if $|x|\leq 1$ and $\eta(x)=0$ if $|x|\geq 2$. For $R>1$, set $\eta_R(x)=\eta(\frac{x}{R})$. Then from \eqref{xepsilon}, we derive
\bal 
\CL_\mu^s(\eta_R|x|^{\tau_+(s,\xm+\xe)})
&=-\xe\eta_R|x|^{\tau_+(s,\xm+\xe)-2s}+
\int_{\BBR^N}\frac{|y|^{\tau_+(s,\xm+\xe)}(\eta_R(x)-\eta_R(y))}{|x-y|^{N+2s}}dy\\ 
&=:-\xe\eta_R|x|^{\tau_+(s,\xm+\xe)-2s}+f_1(x).
\eal
We see that
\bal 
|f_1(x)|\leq c\, R^{-2s+\tau_+(s,\xm+\xe)},\quad\forall x\in\xO\;\;\text{and}\;\;\forall R>4\max\{\diam(\xO),1\}.
\eal
Therefore there exists $R_0=R_0(\xO,s,\xm,\xe)>4\max\{\diam(\xO),1\}$ such that
\ba \label{etax-2}
\CL_\mu^s(\eta_{R_0}|x|^{\tau_+(s,\xm+\xe)}) = -\xe\eta_{R_0}|x|^{\tau_+(s,\xm+\xe)-2s}+f_1(x)\leq -\frac{\xe}{2} |x|^{\tau_+(s,\xm+\xe)-2s},\quad\forall x\in \xO\setminus\{0\}.
\ea
Set $t_0:=2\varepsilon^{-1}\norm{f}_{L^\infty(\xO;|x|^b)}\diam(\xO)^{2s-\tau_+(s,\xm+\xe)-b}$ then from \eqref{etax-2}, we get
\ba \label{etax-3}
\CL_\mu^s(t_0 \eta_R|x|^{\tau_+(s,\xm+\xe)})\leq
 -\norm{f}_{L^\infty(\xO;|x|^b)}\bigg(\frac{\diam(\xO)}{|x|}\bigg)^{2s-\tau_+(s,\xm+\xe)-b}|x|^{-b}, \quad \forall x\in\xO\setminus\{0\}.
\ea
We choose $\xe>0$ small such that $2s-\tau_+(s,\xm+\xe)-b>0$. Set
\bal
\psi(x):=2t_0 (2R_0+1)^{\tau_+(s,\xm+\xe)-\tp}|x|^\tp- t_0\eta_{R_0}(x)|x|^{\tau_+(s,\xm+\xe)},\quad x\in\BBR^N\setminus\{0\}.
\eal
Then by using the definition of $\eta_{R_0}$ and $t_0$, we see that
\bal
 t_0(2R_0+1)^{\tau_+(s,\xm+\xe)-\tp} |x|^\tp \leq \psi(x)\leq  2 t_0(2R_0+1)^{\tau_+(s,\xm+\xe)-\tp} |x|^\tp,
\eal
for any $x\in\BBR^N\setminus\{0\}$.
Moreover, by \eqref{etax-3} and the fact that $\CL_\mu^s(|x|^{\tp})=0$ in $\Omega \setminus \{0\}$, we have
\bal
\CL_\mu^s\psi = -\CL_\mu^s (t_0\eta_R|x|^{\tau_+(s,\xm+\xe)}) \geq
\norm{f}_{L^\infty(\xO;|x|^b)}\bigg(\frac{\diam(\xO)}{|x|}\bigg)^{2s-\tau_+(s,\xm+\xe)-b}|x|^{-b}, \quad \forall x\in\xO\setminus\{0\}.
\eal
This implies
\ba \label{psieq-a}
\ll \psi,\phi \gg_{\mu} \geq \norm{f}_{L^\infty(\xO;|x|^b)}\int_{\xO}\bigg(\frac{\diam(\xO)}{|x|}\bigg)^{2s-\tau_+(s,\xm+\xe)-b}|x|^{-b}\xf(x)dx,\quad\forall0\leq\xf\in C_0^\infty(\xO\setminus\{0\}).
\ea
Set $\overline \psi=|x|^{-\tp}\psi$ and for $0\leq\xf\in C_0^\infty(\xO\setminus\{0\})$, set $\zeta=|x|^{-\tp}\phi$. We deduce from \eqref{psieq-a} that
\ba \label{tildepsieq-a}
\langle \overline \psi, \zeta \rangle_{s,\tp}\geq  \norm{f}_{L^\infty(\xO;|x|^b)}\int_{\xO}\bigg(\frac{\diam(\xO)}{|x|}\bigg)^{2s-\tau_+(s,\xm+\xe)-b}\xz(x)|x|^{-b}|x|^\tp dx.
\ea
Subtracting term by term \eqref{veq-a} and \eqref{tildepsieq-a} yields
\ba\label{7}
\langle  v- \overline \psi, \zeta \rangle_{s,\tp}
\leq 0, \quad \forall \,  0 \leq \xz  \in C_0^\infty(\xO\setminus\{0\}).
\ea
Note that $v \in H_0^s(\Omega;|x|^{\tp})$ and $\norm{\overline{\psi}}_{s,\tp}<+\infty$, hence from  Proposition  \ref{density}, \eqref{7} is valid for any $\xz\in H^{s}_0(\xO;|x|^{\tau_+(s,\xm)})$. In addition, using again Proposition \ref{density}, we may show that $(v-\overline{\psi})^+\in H^{s}_0(\xO;|x|^{\tau_+(s,\xm)})$. By taking $\zeta=(v-\overline{\psi})^+$ in \eqref{7}, we obtain
\bal
\langle  v- \overline \psi, (v- \overline \psi)^+ \rangle_{s,\tp}
\leq 0.
\eal
This implies $(v-\overline{\psi})^+=0$ a.e. in $\BBR^N$, i.e. $v\leq \overline \psi$ a.e. in $\BBR^N$. Therefore
\bal
u(x)\leq c\norm{f}_{L^\infty(\xO;|x|^b)}|x|^{\tp}\quad \;\text{for a.e. } x \in \xO\setminus\{0\},
\eal
where $c=c(N,\Omega,s,\mu,b)$.

Similarly, for $-u$ in place of $u$, we obtain the lower bound, i.e.
\bal
-u(x)\leq c\norm{-f}_{L^\infty(\xO;|x|^b)}|x|^{\tp}\quad \;\text{for a.e. } x \in \xO\setminus\{0\}.
\eal
Combining the above estimates yields \eqref{estLinfty}. The proof is complete.
\end{proof}

\subsection{Regularity}

In the sequel, we assume that $\xO \subset \R^N$ ($N \geq 2$) is an open bounded set containing the origin. Recall that $d(x)=\dist(x,\partial\xO)$.

\begin{lemma}\label{smoothlemma}
Let $b<2s-\tau_+$ and  $f \in L^\infty(\xO;|x|^b)\cap L^2(\xO)$.  Assume that  $u$ is a weak solution of \eqref{solutionL2} such that $u\in \Hm$. Then the following regularity properties hold.

$(i)$ $u\in C^{\xb}_{\loc}(\xO\setminus\{0\})$ for any $\xb\in (0,2s)$. In addition, for any open set $D\Subset \xO\setminus\{0\},$ there exists a positive constant $c$ depending only on $N,\xO,s,\xm, \xb$ and $\dist(D,\partial\xO\cup\{0\})$ such that
\ba\label{smoothest1b}
\norm{u}_{C^\xb(\overline{D})}\leq c\norm{f}_{L^\infty(\xO;|x|^b)}.
\ea

$(ii)$ We additionally assume that $\xO$ satisfies the exterior ball condition. Then for any $\xe>0$ such that $B_{4\xe }(0)\subset \xO$ there exists $c=c(N,\xO,s,\xm,\xb,b,\xe)$ such that
\ba\label{smoothest1}
\norm{d^{-s}u}_{L^\infty(\xO\setminus B_\xe(0))}\leq c\norm{f}_{L^\infty(\xO;|x|^b)}.
\ea

 $(iii)$ In addition, if $\xO$ is $C^{1,1}$ then there exist $0<\alpha<\min\{s,1-s\}$ and $c>0$ depending on  $N,\Omega,s,\mu,b,\varepsilon$  such that
\ba\label{smoothest2}
\norm{d^{-s}u}_{C^\xa(\overline{\xO\setminus B_\xe(0)})}+\norm{u}_{C^s(\overline{\xO\setminus B_{\xe}(0)})}
\leq c\norm{f}_{L^\infty(\xO;|x|^b)}.
\ea

\end{lemma}
\begin{proof}[\textbf{Proof}] (i) Since $u$ is a weak solution of \eqref{solutionL2}, we have
\bal
\ll u,\phi \gg_{\mu} = \int_{\Omega} f \phi dx, \quad \forall \xf\in C_0^\infty(\xO\setminus\{0\}).
\eal
Let $x_0\in \xO\setminus \{0\},$ then there exists $\xr>0$ such that $B_{2\xr}(x_0)\subset \xO\setminus \{0\}$. Let $\{\xz_\xd\}$ be the sequence of standard mollifiers, then for $\xd>0$ small enough, the function $u_\xd=\xz_\xd*u$ satisfies
\ba \label{eq-udelta}
(-\xD)^su_{\xd}=H_\xd \quad \text{in } B_{\xr}(x_0),
\ea
where 
\ba \label{Hdelta} H_\xd:=\xz_\xd*\left(f-\mu \frac{u}{|x|^{2s}}\right).
\ea
By \eqref{estLinfty}, we have that $u_\xd\in C^\infty(\BBR^N)$ and there exists a positive constant $c$ depending on $N,s,\xm,B_{\xr}(x_0),\xO,b$ such that
\bal 
|H_\xd(x)| +|u_\xd(x)|\leq c\norm{f}_{L^\infty(\xO;|x|^b)}, \quad \forall x \in B_\rho(x_0).
\eal
Furthermore, by \eqref{estLinfty}, we deduce that
\bal 
\int_{\BBR^N}\frac{|u_\xd(x)|}{1+|x|^{N+2s}}dx\leq c\norm{f}_{L^\infty(\xO;|x|^b)}.
\eal
Hence, by \cite[Corollary 2.5]{RS}, for any $\xb\in(0,2s)$, there holds
\ba \label{smooth-0}
\norm{u_{\xd}}_{C^\xb(\overline{B_{\frac{\xr}{4}}}(x_0))}\leq c\norm{f}_{L^\infty(\xO;|x|^b)},
\ea
where the constant $c>0$ depends only on $N,s,\xm, x_0,\xO,\xb,b.$  By the Arzel\'a-Ascoli theorem, there exists a sequence $\{\xd_n\}$ converging to $0$ such that $u_{\xd_n}\to u$ in $C^{\xb}(\overline{B_{\frac{\xr}{16}}}(x_0)),$ hence
\ba\label{estsmooth}\BAL
\norm{u}_{C^\xb(\overline{B_{\frac{\xr}{16}}}(x_0))}&\leq c\norm{f}_{L^\infty(\xO;|x|^b)}
\EAL
\ea
for any $\xb\in (0,2s)$. By the standard covering argument, we derive $u \in C_{\loc}^{\beta}(\Omega \setminus \{0\})$ and estimate \eqref{smoothest1b}.

(ii) Let $0\leq\eta\leq1$ be a smooth function such that $\eta=0$ for any $|t|\leq \frac{1}{2}$ and $\eta=1$ for any $|t|\geq 1.$ Set $\eta_{\xe}(x)=\eta(\xe^{-1} |x|)$ for any $\xe>0.$ If $\xe$ is small enough such that $B_{4\xe}(0)\subset \xO,$ we can easily show that the function $\eta_{\varepsilon}u$ satisfies
\bal
\ll \eta_{\varepsilon}u, \phi \gg_{\mu} = \int_{\Omega} \eta_{\varepsilon}f \phi dx + \int_{\Omega}\varphi \phi dx,\quad\forall \xf\in C_0^\infty(\xO\setminus B_{\frac{\xe}{2}}(0)),
\eal
where
\bal
\varphi(x):=C_{N,s}\lim_{\xd\to0}\int_{\BBR^N\setminus B_\xd(x)}\frac{(\eta_\xe(x)-\eta_\xe(y))u(y)}{|x-y|^{N+2s}}dy,\quad\forall x\in\xO\setminus B_{\frac{\xe}{2}}(0).
\eal

 Now, for any $x\in \xO\setminus B_{\frac{\xe}{2}}(0)$, we have
 \bal
 \varphi(x)&= C_{N,s}\lim_{\xd\to0}\int_{\BBR^N\setminus B_\xd(x)}\frac{(\eta_\xe(x)-\eta_\xe(y))(u(y)-u(x))}{|x-y|^{N+2s}}dy\\
 &\quad +u(x)C_{N,s}\lim_{\xd\to0}\int_{\BBR^N\setminus B_\xd(x)}\frac{\eta_\xe(x)-\eta_\xe(y)}{|x-y|^{N+2s}}dy.
 \eal
By \eqref{estLinfty} and \eqref{estsmooth}, we may show that
\bal
|\varphi(x)| \leq c\norm{f}_{L^\infty(\xO;|x|^b)}, \quad \forall x \in \xO\setminus B_{\frac{\xe}{2}}(0).
\eal
Set \bal h(x):=f(x)\eta_\xe(x)-\xm\frac{u\eta_\xe}{|x|^{2s}}+ \varphi(x),
\eal
then there exists a positive constant $c=c(N,\xO,s,\xm,\xe)$ such that
\ba \label{h-1}
|h(x)|\leq  c\norm{f}_{L^\infty(\xO;|x|^b)},\quad\forall x\in\xO\setminus B_{\frac{\xe}{2}}(0).
\ea
Using \cite[Lemma 2.7]{RS},  taking into account that $\Omega$ satisfies the exterior ball condition (which is needed to apply \cite[Lemma 2.7]{RS})  and \eqref{h-1}, we obtain \eqref{smoothest1}.

(iii) By \cite[Proposition 1.1 and Theorem 1.2]{RS}, the assumption that $\Omega$ is a $C^{1,1}$ bounded domain (which is required to apply \cite[Theorem 1.2]{RS})  and \eqref{h-1}, we derive \eqref{smoothest2}.
\end{proof}

The next result provides a higher  H\"older regularity of weak solutions to equation \eqref{solutionL2}.

\begin{lemma}\label{smoothlemma2}
Let $b<2s-\tau_+$, $\theta\in(0,1)$ and $f \in L^\infty(\xO;|x|^b)\cap C^\theta_{\loc}(\xO\setminus\{0\})$.  Assume that  $u$ is a weak solution of \eqref{solutionL2} such that $u \in \Hm$. Then $u\in C^{\xb_0+2s}_{\loc}(\xO\setminus\{0\})$ for some $\xb_0>0.$ Furthermore for any open set $D\Subset \xO\setminus\{0\},$ there exists a positive constant $c$ depending only on $N,s,\xO,\xm,\xb_0$ and $\dist(D,\partial\xO\cup\{0\})$ such that
\ba\label{smoothest3}
\norm{u}_{C^{2s+\xb_0}(\overline{D})}\leq c (\norm{f}_{L^\infty(\xO;|x|^b)}+\norm{f}_{C^{\theta}(\overline{D})}).
\ea
\end{lemma}
\begin{proof}[\textbf{Proof}]
First we note that
\bal
\ll u,\phi \gg_{\mu} = \int_{\Omega} f \phi dx, \quad \forall \xf\in C_0^\infty(\xO\setminus\{0\}).
\eal

Let $x_0\in \xO\setminus \{0\}$ then there exists $\xr>0$ such that $B_{2\xr}(x_0)\subset \xO\setminus \{0\}$. Consider the mollifiers $\xz_\xd$, then for $\xd>0$ small enough and put $u_\xd=\xz_\xd \ast u$. {Then $u_\delta$ solves equation \eqref{eq-udelta} with $H_\delta$ as in \eqref{Hdelta}. 
By repeating the argument after \eqref{Hdelta} in part (i) of the proof of Lemma \ref{smoothlemma}, we obtain \eqref{smooth-0}.}
This, together with \cite[Corollary 2.4]{RS}, implies the existence of a constant $0<\xb_0<\min\{\theta,2s\}$ such that
\bal
\norm{u_{\xd}}_{C^{2s+\xb_0}(\overline{B_{\frac{\xr}{16}}}(x_0))}\leq c\,\big(\norm{f}_{C^\theta(\overline{B_{\xr}}(x_0))}+\norm{f}_{L^\infty(\xO;|x|^b)}\big),
\eal
where the constant $c$ depends only on $N,s,\xm,|x_0|, d(x_0), \xO,\xb_0.$ By the Arzel\'a--Ascoli theorem, there exists a subsequence $\{\xd_n\}$ converging to $0$ such that $u_{\xd_n}\to u$ in $C^{2s+\xb_0}(\overline{B_{\frac{\xr}{16}}}(x_0))$, hence
\ba\label{estsmoothb}\BAL
\norm{u}_{C^{2s+\xb_0}(\overline{B_{\frac{\xr}{16}}}(x_0))}&\leq c\,\big(\norm{f}_{C^\theta(\overline{B_{\xr}}(x_0))}+\norm{f}_{L^\infty(\xO;|x|^b)}\big).
\EAL
\ea
The desired results follows by the above inequality and a standard covering argument.
\end{proof}

As a consequence of the above results, we obtain the following result

\begin{corollary}\label{rem 3.1-2}
 Assume that   $b<2s-\tau_+,$ $\theta\in(0,1)$ and $f\in L^\infty(\xO;|x|^b)\cap C^\theta_{\loc}(\xO\setminus\{0\})$. Then the weak solution of \eqref{veryweaksolution} belongs to $\mathbf{X}_\xm(\xO;|x|^{-b})$.
\end{corollary}

\section{Poisson problems} \label{sec:Poisson}

In this section, we study problem \eqref{veryweaksolutionmesure}. 
\subsection{$L^1$ sources} \label{subsec:L1source}

We start with the case of $L^1$ source.

\begin{proof}[\textbf{Proof of Theorem \ref{existence2}}]
\noindent \textit{Existence.}
Let $\theta\in(0,1)$ and $\{f_n\} \subset C^\theta(\overline{\xO})$ such that $f_n\to f$ in $L^1(\xO; d(x)^s|x|^{\tau_+})$. By Proposition \ref{existence}, there exists a unique function $v_n\in H^{s}_0(\xO;|x|^{\tau_+})$ such that
\ba \label{vngn-1}
\langle v_n,\phi \rangle_{s,\tau_+} = \int_{\Omega}f_n \phi |x|^{\tau_+}dx, \quad \forall \xf\in H^{s}_0(\xO;|x|^{\tau_+}).
\ea
Since $v_n\in H^{s}_0(\xO;|x|^{\tau_+}),$ by Proposition \ref{density}, there exists $v_{n,m}\in C_0^\infty(\xO\setminus\{0\})$ such that
\bal
\lim_{m\to\infty}\norm{v_n-v_{n,m}}_{s,\tau_+}=0.
\eal

Let $\psi\in \mathbf{X}_\xm(\xO;|x|^{-b})$. By property (iii) in the definition of $\mathbf{X}_\xm(\xO;|x|^{-b})$ (see Definition \ref{def:weaksol}), we have
\ba\label{check}
\langle v_{n,m}, \psi \rangle_{s,\tau_+}
=\int_\xO v_{n,m}(x)(-\xD)^s_{\tau_+}\psi(x) |x|^{\tau_+}dx.
\ea
Since $|(-\xD)^s_{\tau_+}\psi(x)|\leq C|x|^{-b}$ for a.e. $x \in \Omega \setminus \{0\}$ due to property (ii) in Definition \ref{def:weaksol}, by \eqref{subcrit1} and \eqref{crit1}, we have
\ba \label{vnmvn}
\lim_{m\to\infty}\int_\xO v_{n,m}(x)(-\xD)^s_{\tau_+}\psi(x) |x|^{\tau_+}dx=\int_\xO v_{n}(x)(-\xD)^s_{\tau_+}\psi(x) |x|^{\tau_+}dx.
\ea
Combining \eqref{vngn-1}--\eqref{vnmvn} yields
\ba  \label{vnfn-2}
\int_\xO v_{n}(x)(-\xD)^s_{\tau_+}\psi(x) |x|^{\tau_+}dx = \langle v_n,\psi \rangle_{s,\gamma} =
\int_{\xO}f_n(x)\psi(x) |x|^{\tau_+}dx.
\ea
By Lemma \ref{smoothlemma} (i), $v_n\in C^{\xb}(\xO\setminus\{0\})$ for any $\xb\in (0,2s)$.

Let $h_m:\BBR\to\BBR$ be a smooth function such that $h_m(t)\to \sign(t)$  and $|h_m(t)|\leq 1$ for any $t\in \BBR$. By Proposition \ref{existence} and Remark \ref{LinftyL2}, there exists a unique weak solution $U_m \in \Hm$  of
\bal \CL_\mu^s U=|x|^{-b} h_m(v_n) \quad \text{in } \Omega
\eal
in the sense of Definition \ref{weaksolution}. Put $\tilde U_m=|x|^{-\tau_+} U_m$, then $\tilde U_m \in H_0^s(\Omega;|x|^{\tau_+})$ and
\ba \label{tildeUzeta}
\langle \tilde U, \zeta \rangle_{s,\tau_+}
= (h_m(v_n),\zeta)_{\tau_+-b}, \quad \forall \zeta \in H_0^s(\Omega;|x|^{\tau_+}).
\ea
By \eqref{estLinfty} and \eqref{smoothest1}, we have
\ba\label{estLinfty2}
|\tilde U_m(x)|\leq c\,d(x)^s \quad \text{for a.e. } x \in \xO,
\ea
where the constant $c$ is independent of $\tilde U_m$. Furthermore in view of Lemma \ref{smoothlemma2}, we can easily show that $\tilde U_m\in \mathbf{X}_\xm(\xO;|x|^{-b})$. By taking $\psi=\tilde U$ in \eqref{vnfn-2}, we derive
\bal
\int_{\Omega} v_n(x) (-\Delta)^s_{\tau_+}\tilde U(x) |x|^{\tau_+} dx
= \langle v_n, \tilde U \rangle_{s,\tau_+}
=\int_{\Omega}f_n(x)\tilde U(x)|x|^{\tau_+} dx.
\eal
Taking $\zeta=v_n$ as a test function in \eqref{tildeUzeta}, we have
\bal
\langle  \tilde U, v_n \rangle_{s,\tau_+} =  \int_{\xO}h_m(v_n)v_n |x|^{\tau_+-b} dx.
\eal
From the two preceding equalities and \eqref{estLinfty2}, we derive
\bal
\int_\xO v_n(x)h_m(v_n(x)) |x|^{\tau_+-b}dx\leq c\int_{\xO}|f_n(x)| d(x)^s |x|^{\tau_+}dx.
\eal
By \eqref{subcrit1}, \eqref{crit1}, the fact that $h_m(t)\to \sign(t)$ as $m \to \infty$, and the dominated convergence theorem, we have
 \ba\label{4.14}
\int_\xO |v_n(x)| |x|^{\tau_+-b}dx\leq c\int_{\xO}|f_n(x)|d(x)^s |x|^{\tau_+}dx.
 \ea
Using a similar argument, we can show that
\ba \label{vnvk-1}
\| v_n - v_k \|_{L^1(\Omega;|x|^{\tau_+-b})} \leq c\, \| f_n-f_k \|_{L^1(\Omega;d(x)^s|x|^{\tau_+})}, \quad \forall n,k \in \N.
\ea
Since $\{f_n\}$ is a convergent sequence in $L^1(\Omega;d(x)^s|x|^{\tau_+})$, we deduce from \eqref{vnvk-1} that $\{v_n\}$ is a Cauchy sequence in $L^1(\Omega;|x|^{\tau_+-b})$, which in turn implies that there exists $v \in L^1(\Omega;|x|^{\tau_+-b})$ such that $v_n \to v$ in $L^1(\Omega;|x|^{\tau_+-b})$. By property (ii) in Definition \ref{def:weaksol}, \eqref{estLinfty3} and the dominated convergence theorem, we deduce from \eqref{vnfn-2} that
\bal
\int_\xO v(-\xD)^s_{\tau_+}\psi |x|^{\tau_+}dx=\int_{\xO}f(x)\psi(x) |x|^{\tau_+}dx.
\eal
Put $u=|x|^{\tau_+} v$ then $u \in L^1(\Omega;|x|^{-b})$ and $u$ is a weak solution of \eqref{veryweaksolution}. \medskip

\noindent \textit{A priori estimate.} By letting $n \to \infty$ in \eqref{4.14} and using $u=|x|^{\tau_+} v$, we deduce \eqref{est:apriori-1}. \medskip

\noindent \textit{Kato type inequality.} First we prove estimate \eqref{Kato:+-1}. Let $\{\xO_l\}_{l \in \N}$ be a smooth exhaustion of $\xO$, namely $\Omega_l \Subset \Omega_{l+1}$ and $\cup_{l \in \N}\Omega_l = \Omega$.  Let $\phi_l\in C^\infty_0(\xO)$ be such that $\phi_l=1$ on $\overline{\xO}_l$ and $0\leq\phi_l\leq1$. Put $\xi_l = (-\xD)^s_{\tau_+}\phi_l$.

For $\varepsilon>0$, since $(v_n-\xe\xf_l)^+\in H^{s}_0(\xO;|x|^{\tau_+}),$ by Theorem \ref{density}, there exists a sequence $0\leq \tilde v_{n,m}\in C_0^\infty(\xO\setminus\{0\}),$ such that
\bal
\lim_{m\to\infty}\norm{(v_n-\xe\phi_l)^+-\tilde v_{n,m}}_{H^{s}_0(\xO;|x|^{\tau_+})}=0.
\eal
Proceeding as above, for any $\psi \in \mathbf{X}_\xm(\xO;|x|^{-b})$, we have that
\bal
\langle (v_n-\varepsilon \phi_l)^+, \psi \rangle_{s,\tau_+} &= \lim_{m \to \infty} \langle \tilde v_{n,m}, \psi  \rangle_{s,\tau_+} \\
&=\lim_{m\to\infty}\int_\xO \tilde v_{n,m}(-\xD)^s_{\tau_+}\psi |x|^{\tau_+}dx=\int_\xO (v_n-\xe\xf_l)^+(-\xD)^s_{\tau_+}\psi |x|^{\tau_+}dx.
\eal

By Theorem \ref{th:main-2}, there holds
\ba \label{vnezetal} \begin{split}
\int_\xO (v_{n}-\xe \phi_l)^+&(-\xD)^s_{\tau_+}\psi |x|^{\tau_+}dx\leq\int_{\xO}\sign^+(v_{n}-\xe\phi_l )\psi(f_n-\xe \xi_l)|x|^{\tau_+}dx\\
&\leq\int_{\xO_l}\sign^+(v_{n}-\xe )\psi(f_n-\xe \xi_l)|x|^{\tau_+}dx+\int_{\xO\setminus\xO_l}\psi|f_n-\xe \xi_l||x|^{\tau_+}dx,
\end{split}
\ea
for any nonnegative $\psi\in \mathbf{X}_\xm(\xO;|x|^{-b})$.

Let $\xe_m\downarrow 0$ be such that $|\{x\in\xO:\;v=\xe_m\}|=0,$ then $\sign^+(v_{n}-\xe_m )\to \sign^+(v-\xe_m )$ a.e. in $\xO$. Replacing $\varepsilon$ by $\varepsilon_n$ in \eqref{vnezetal} and using the estimate $|\psi(x)| \leq c\,d(x)^s$ for a.e. $x \in \Omega$, we obtain
\bal
	\int_\xO (v_{n}-\xe_m \phi_l)^+(-\xD)^s_{\tau_+}\psi |x|^{\tau_+}dx
	&\leq\int_{\xO_l}\sign^+(v_{n}-\xe_m )\psi(f_n- \xe_m  \xi_l)|x|^{\tau_+}dx \\
	&\quad +c\int_{\xO\setminus\xO_l}|f_n-\xe_m \xi_l|d(x)^s|x|^{\tau_+}dx.
\eal
Letting $n\to\infty$ and  $m\to\infty$ in \eqref{vnezetal} successively, we obtain
\bal
\int_\xO v^+(-\xD)^s_{\tau_+}\psi |x|^{\tau_+}dx\leq\int_{\xO_l}\sign^+(v)\psi f|x|^{\tau_+}dx+c\int_{\xO\setminus\xO_l}|f|d(x)^s|x|^{\tau_+}dx.
\eal
Using the relation $u=|x|^{\tau_+}v$ and  letting $l \to +\infty$ in the above estimate yield \eqref{Kato:+-1}.

By employing \eqref{Kato:+-1} with $-v$ in place of $v$, we obtain
\bal
\int_\xO (-v)^+(-\xD)^s_{\tau_+}\psi |x|^{\tau_+}dx\leq-\int_{\xO}\sign^+(-v)\psi f|x|^{\tau_+}dx.
\eal
Adding term by term of the above inequality and \eqref{Kato:+-1}, and then using the relation $u=|x|^{\tau_+}v$, we  deduce \eqref{Kato||-1}. \medskip

\noindent \textit{Monotonicity and uniqueness.}  Let $f_1,f_2 \in L^1(\Omega;d(x)^s|x|^{\tau_+})$ such that $f_1 \leq f_2$ a.e. in $\Omega$ and let $u_i$, $i=1,2$, is a weak solution to \eqref{veryweaksolution} with $f=f_i$. By \eqref{Kato:+-1} with $u$ replaced by $u_1-u_2$ and $f$ replaced by $f_1-f_2$, we obtain
\ba \label{Katou1u2}
\int_\xO (u_1-u_2)^+(-\xD)^s_{\tau_+}\psi dx\leq\int_\xO (f_1-f_2)\sign^+(u_1-u_2)\psi |x|^{\tau_+} dx \leq 0,\ \ \forall \, 0\leq\psi\in \mathbf{X}_\xm(\xO;|x|^{-b}).
\ea
Let $\xi_b$ be the weak solution to
\ba \label{xib} \left\{  \BAL
\CL_\mu^s \xi &= |x|^{-b} \  &&\text{in } \Omega, \\
\xi &=0 \  &&\text{in }  \R^N \setminus \Omega.
\EAL \right. \ea
Then $\xi_b \in \mathbf{X}_\xm(\xO;|x|^{-b})$ due to Corollary \ref{rem 3.1-2} and hence by \eqref{estLinfty3}, $|\xi_b| \leq Cd^s$ a.e. in $\Omega$. Taking $\psi=\xi_b$ in \eqref{Katou1u2}, we deduce that $(u_1-u_2)^+ =0$ a.e. in $\Omega \setminus \{0\}$, namely $u_1 \leq u_2$ a.e. in $\Omega \setminus \{0\}$. Thus the mapping $f \mapsto $

The uniqueness follows straightforward from the monotonicity. 

 In fact, let $ f \in L^1(\Omega;d(x)^s|x|^{\tau_+})$ be nonnegative and $u$ be the unique weak solution to problem \eqref{veryweaksolution}.  Since the mapping $f \mapsto u$ is nondecreasing, $f \geq 0$, and $0$ is the unique weak solution to problem \eqref{veryweaksolution} with zero source, we deduce that $u \geq 0$ a.e. in $\Omega \setminus \{0\}$. 
\end{proof}

\subsection{Measure sources} \label{subsec:measuresource}

We start this subsection by noting that if $\xm_0<\xm$ then
\bal
\frac{C_{N,s}}{2}\int_{\BBR^N}\int_{\BBR^N}\frac{|u(x)-u(y)|^2}{|x-y|^{N+2s}}dydx+\xm \int_{\BBR^N}\frac{|u|^2}{|x|^{2s}}dx &\geq\frac{C_{N,s}}{2}\frac{\xm_0-\xm}{\xm_0}\int_{\BBR^N}\int_{\BBR^N}\frac{|u(x)-u(y)|^2}{|x-y|^{N+2s}}dydx\\
&\geq C(N,s,\xm,\xm_0)\left(\int_{\BBR^N}|u|^\frac{2N}{N-2s}dx\right)^{\frac{N-2s}{N}},
\eal
for any $u\in C^\infty_0(\xO)$. 
Setting $u=|x|^{\tau_+}v$, we derive
\ba\label{subcritsobolev1}
\BAL
\frac{C_{N,s}}{2}\int_{\BBR^N}\int_{\BBR^N}\frac{|v(x)-v(y)|^2}{|x-y|^{N+2s}}|y|^{\tau_+}dy |x|^{\tau_+}dx\geq C(\xO,N,s,\xm,\xm_0)\left(\int_{\Omega}|v|^\frac{2N}{N-2s}|x|^{\frac{2N}{N-2s}\tau_+}dx\right)^{\frac{N-2s}{N}}.
\EAL
\ea
By Theorem \ref{density},  inequality \eqref{subcritsobolev1} is valid for any $v\in H^{s}_0(\xO;|x|^{\tau_+})$.

If $\xm=\xm_0$ then by \cite[Theorem 3]{tz}, there exists a positive constant $C=C(N,s)$ such that
\bal
\frac{C_{N,s}}{2}\int_{\BBR^N}\int_{\BBR^N}\frac{|u(x)-u(y)|^2}{|x-y|^{N+2s}}dydx+\xm_0 \int_{\BBR^N}\frac{|u|^2}{|x|^{2s}}dx\geq C(N,s) \left(\int_{\BBR^N}|u|^\frac{2N}{N-2s}X(|x|)^{\frac{2(N-s)}{N-2s}}dx\right)^{\frac{N-2s}{N}},
\eal
for any $u\in C^\infty_0(\xO)$, where $X$ is defined in \eqref{X}. Setting $u(x)=|x|^{\frac{2s-N}{2}}v(x),$  we have 
\ba\label{critsobolev1}\BAL
\int_{\BBR^N}\int_{\BBR^N}\frac{|v(x)-v(y)|^2}{|x-y|^{N+2s}}|y|^{\frac{2s-N}{2}}dy |x|^{\frac{2s-N}{2}}dx\geq C(N,s) \left(\int_{\BBR^N}|v|^\frac{2N}{N-s}|x|^{-N}X(|x|)^{\frac{2(N-s)}{N-2s}}dx\right)^{\frac{N-2s}{N}}.
\EAL
\ea
By Theorem \ref{density}, inequality \eqref{critsobolev1} is valid for any $v\in H^{s}_0(\xO;|x|^{\frac{2s-N}{2} })$.

From \eqref{subcritsobolev1} and \eqref{critsobolev1}, we see that, for any $\mu \geq \mu$, 
\ba\label{subcritsobolev2}
\BAL
&\quad \frac{C_{N,s}}{2}\int_{\BBR^N}\int_{\BBR^N}\frac{|v(x)-v(y)|^2}{|x-y|^{N+2s}}|y|^{\tau_+}dy |x|^{\tau_+}dx\\
&\geq C(\xO,N,s,\xm,\xm_0)\left(\int_{\Omega}|v|^\frac{2N}{N-2s}|x|^{\frac{2N}{N-2s}\tau_+}X(|x|)^{\frac{2(N-s)}{N-2s}}dx\right)^{\frac{N-2s}{N}}, \quad \forall v \in H^{s}_0(\xO;|x|^{\tau_+}). 
\EAL
\ea

We need the following a priori Lebesgue estimate on weak solutions to equation \eqref{solutionL2}. 

We recall the definition of weak Lebesgue spaces. Let $D$ be a bounded domain in $\R^N$. The weak Lebesgue space $L_w^q(D)$, $1 \le q < \infty$, is defined by
\bal
L_w^q(D):= \left\{ u \in L^1_{\loc}(D):  \sup_{\lambda > 0} \lambda^q \int_{D} \mathbf{1}_{\{ x \in D: |u(x)| > \lambda \} } dx < +\infty \right\},
\eal
where $\mathbf{1}_{A}$ denotes the indicator function of a set $A \subset \R^N$. Put
\bal
\|u\|^*_{L_w^q(D)}:= \left( \sup_{\lambda > 0} \lambda^q \int_{\Omega} \mathbf{1}_{\{ x \in D: |u(x)| > \lambda \} } dx \right)^{\frac{1}{q} }.
\eal
We note that $\|\cdot\|^*_{L_w^q(D)}$ is not a norm, but for $q >1$, it is equivalent to the norm
\bal
\|u\|_{L_w^q(D)}:= \sup \left\{ \frac{\int_{A} |u|^q dx }{ (\int_{A} dx)^{1-\frac{1}{q}} }: A \subset \Omega, A \text{ measurable},  |A|>0\right\}.
\eal
In fact, there hold 
\ba \label{Marcin-equi}
	\|u\|^*_{L_w^q(D)} \leq \|u\|_{L_w^q(D)} \leq \frac{q}{q-1}\|u\|^*_{L_w^q(D)}, \quad \forall u \in L_w^q(D).
\ea


It is well-known that the following embeddings hold
\bal L^q(D) \subset L_w^q(D) \subset L^{r}(D),\quad \forall r \in [1,q).
\eal

\begin{lemma}\label{est2}
 Assume that  $\mu \geq \mu_0$, $f\in L^2(\xO)$ and $u \in \Hm$ is the unique weak solution of \eqref{solutionL2}. For any $r>0$ and $q\in(1, 2_s^* )$, there exists $C=C(N,\Omega,s,\mu,r,q)>0$ such that
	\ba \label{u>k-1}
	\| u \|_{L^q(\Omega \setminus B_r(0))} \leq C \| f \|_{L^1(\Omega;|x|^{\tau_+})}.
	\ea
\end{lemma}
\begin{proof}[\textbf{Proof}] Since $f \in L^2(\Omega)$, the existence and uniqueness of a weak solution $u \in \Hm$ of \eqref{solutionL2} is guaranteed by Proposition \ref{existence}.
Put $v=|x|^{-\tau_+}u$, then $v\in H^{s}_0(\xO;|x|^{\tau_+})$ by Theorem \ref{th:main-1} (ii) and
\ba \label{vphi-a}
	\langle v,\phi \rangle_{s,\tau_+}
	=\int_{\xO}f(x)\xf(x)|x|^{\tau_+}dx, \quad \forall \xf\in H^{s}_0(\xO;|x|^{\tau_+}).
\ea
Since $f \in L^2(\Omega)$ and $\tau_+ \geq \frac{2s-N}{2}$, by the H\"older inequality, we deduce that $f \in L^1(\Omega;|x|^{\tau_+})$. 

Let $\lambda>0,$ taking $v_\lambda :=\max\{-\lambda,\min \{v,\lambda\}\}$ as a test function in \eqref{vphi-a}, we have
\bal
	\langle v,v_\lambda \rangle_{s,\tau_+}
	&=\int_{\xO} f(x) v_\lambda(x)|x|^{\tau_+} dx\leq \lambda\int_{\xO} |f(x)||x|^{\tau_+} dx.
\eal
	We see that
	\bal
	(v(x)-v(y))(v_\lambda(x)-v_\lambda(y))\geq (v_\lambda(x)-v_\lambda(y))^2, \quad \forall x,y \in \R^N.
	\eal
	Hence from the  two proceeding inequalities, we obtain
	\bal
	\langle v_\lambda,v_\lambda \rangle_{s,\tau_+}
	\leq \lambda \int_{\xO}|f(x)||x|^{\tau_+} dx.
	\eal
 Therefore for $r>0$ such that $B_{4r}(0) \subset \Omega$,
\bal
\{x\in\xO\setminus B_r(0):\;|u(x)|\geq \lambda\} &= \{x\in\xO\setminus B_r(0):\;|v(x)|\geq \lambda |x|^{-\tau_+}\} \\
&\subset  \{x\in\xO\setminus B_r(0):\;|v(x)|\geq a\lambda \} \\
&\subset \{x\in\xO\setminus B_r(0):\;|v_\lambda(x)|\geq \min\{a,1\}\lambda \},
\eal
where $a=a(N,\Omega,s,\mu,r)$.
This and \eqref{subcritsobolev2} imply	
	\bal
	|\{x\in\xO&\setminus B_r(0):\;|u(x)|\geq \lambda\}|\leq |\{x\in\xO\setminus B_r(0):\;|v_\lambda(x)|\geq \min\{a,1\} \lambda \}|\\
	&\leq C(N,s,\mu,r)\lambda^{-2_s^*}\int_{\xO\setminus B_r(0)}|v_\lambda|^{2_s^*} dx\\
	&\leq C(N,s,\mu,\xO,r)\lambda^{-2_s^*}\int_{\xO}|v_\lambda|^{2_s^*} |x|^{2_s^*\tau_+}X(|x|)^{\frac{2(N-s)}{N-2s}} dx\\
	&\leq C(N,s,\mu,\xO,r)\lambda^{-2_s^*}\left(\int_{\BBR^N}\int_{\BBR^N}\frac{(v_\lambda(x)-v_\lambda(y))^2}{|x-y|^{N+2s}}|y|^{\tau_+}dy |x|^{\tau_+}dx\right)^{\frac{N}{N-2s}}\\
	&\leq C(N,s,\mu,\xO,r)\lambda^{-2_s^*}\left(\int_{\xO}|f(x)||x|^{\tau_+}dx\right)^{2_s^*}, \quad \forall \lambda>0.
	\eal
Therefore,
\bal
\| u \|_{L_w^{2_s^*}(\Omega \setminus B_r(0))} \leq C(N,\Omega,s,\mu,r) \| f \|_{L^1(\Omega;|x|^{\tau_+})},
\eal
which leads to \eqref{u>k-1} due to the continuous embedding $L^q(\Omega \setminus B_r(0)) \subset L^{2_s^*}(\Omega \setminus B_r(0))$ for any $q \in (1,2_s^*)$.
\end{proof}

First we will treat the case where the measure source is concentrated in $\Omega \setminus \{0\}$.


\begin{proof}[\textbf{Proof of Theorem \ref{inter}}]
	
\noindent \textit{Existence.}	By the linearity, we may assume that $\xn\geq 0.$ First we assume that $\xn$ has compact support in $\xO\setminus\{0\}.$ Consider a sequence of standard  mollifiers $\{\xz_{\delta} \}_{\delta >0}$ such that $\xn_{\delta}=\xz_{\delta} \ast \xn\in C^\infty_0(D),$ where $D\Subset\xO\setminus\{0\}$ is an open set. Then
\ba \label{nuntonu}
\int_{\xO} d(x)^s|x|^{\tau_+}\xn_{\delta} dx\to \int_{\xO\setminus\{0\}} d(x)^s|x|^{\tau_+}d\xn \quad \text{as } \delta \to 0
\ea
and
\ba \label{nun<nu}
\| \xn_{\delta} \|_{L^1(\Omega;|x|^{\tau_+})} \leq c \| \nu \|_{\mathfrak{M}(\Omega\setminus\{0\};d(x)^s|x|^{\tau_+})}, \quad \forall \delta > 0.
\ea

Let $u_\delta$ be the nonnegative weak solution of \eqref{veryweaksolution} with $f=\xn_{\delta}$, namely $u_\delta \in L^1(\xO;|x|^{-b})$ and
\ba \label{solun-1}
\int_\xO u_\delta(-\xD)^s_{\tau_+}\psi dx=\int_\xO \psi |x|^{\tau_+} \xn_{\delta} dx,\quad\forall\psi\in \mathbf{X}_\xm(\xO;|x|^{-b}).
\ea
 In view of the proof of Theorem \ref{existence2}, $u_\delta\in \Hm$ and there holds
\bal
\ll u_\delta,\phi \gg_{\mu}=\int_{\xO}\xn_{\delta}\xf \, dx, \quad \forall \xf\in C_0^\infty(\xO\setminus\{0\}).
\eal
Furthermore,  by \eqref{est:apriori-1} and \eqref{nun<nu}, for any $n\in \BBN$, there holds
\ba\label{wheightineq}
\| u_\delta \|_{L^1(\Omega;|x|^{-b})}
\leq c \| \xn_{\delta} \|_{L^1(\Omega;d(x)^s|x|^{\tau_+})} \leq  c \| \nu \|_{\mathfrak{M}(\Omega\setminus\{0\};d(x)^s|x|^{\tau_+})},
 \ea
where $c=c(N,\Omega,s,\mu,b) >0$.

Let $x\in\xO$ and $r>0$ be such that $B_{4r}(0)\subset\xO\setminus\{0\}.$ We assume that $\xe>0$ is small enough and we set $u_{\xd,\xe}=\xz_{\xe} \ast u_\xd\in C^\infty_0(\BBR^N).$ In addition, $u_{\xd,\xe}$ satisfies
\bal \left\{ \BAL
(-\xD)^s u&=h_\xe&&\quad\text{in}\;B_{r}(x)\\
u&=u_{\xd,\xe}&&\quad\text{in}\;\BBR^N\setminus B_{r}(x),
\EAL \right.
\eal
where $h_\xe=\xz_\xe*(\xn_\xd-\mu \frac{u_{\xd}}{|x|^{2s}}).$ Set $v_\xe=u_{\xd,\xe}-\tilde u_{\xd,\xe},$ where $\tilde u_{\xd,\xe}$ is a solution of
\bal \left\{ \BAL
(-\xD)^s u&=0&&\quad\text{in}\;B_{r}(x)\\
u&=u_{\xd,\xe}&&\quad\text{in}\;\BBR^N\setminus B_{r}(x).
\EAL \right.
\eal
By \cite[Proposition 2.5]{CV}, for any $p\in (1,\frac{N}{N-2s})$ and $\xg>\frac{N}{p'}$, there holds
\ba \nonumber 
\norm{v_\xe}_{W^{2s-\xg,p}(B_{r}(x))} &\leq C(r,p,\xg)\int_{B_r(x)} h_\xe dx \\ \label{ineq1} 
&\leq C(r,p,\xg)\int_{ \Omega}|\xn_\xd|dx\, \leq C(r,p,\xg) \| \nu \|_{\mathfrak{M}(\Omega\setminus\{0\};d(x)^s|x|^{\tau_+})},
\ea
where in the last two inequalities we have used \eqref{wheightineq}.

By \cite[Theorem 2.10]{Buc}, we have that 
\bal
\tilde u_{\xd,\xe}(y)=c(N,s)\int_{\BBR^N\setminus B_r(x)}\left(\frac{r^2-|y|^2}{|z|^2-r^2}\right)^{s}|y-z|^{-N}u_{\xd,\xe}(z)dz,\quad\forall y\in B_r(x).
\eal
By \eqref{wheightineq} and the above equality, we can easily show that 
\ba\label{ineq2}
\norm{|\nabla \tilde u_{\xd,\xe}|}_{L^\xk(B_{\frac{r}{2}}(x))}\leq C(r,\xk,N,s)  \| \nu \|_{\mathfrak{M}(\Omega\setminus\{0\};d(x)^s|x|^{\tau_+})},\quad\forall\xk>1.
\ea

Taking into account \eqref{ineq1}, \eqref{ineq2} and using a standard covering argument, we can easily show that for any open set $D\Subset\xO\setminus \{0\}$ there exists a positive constant $C=C(r,\xk,N,s)$ such that
\ba\label{ineq3}
\norm{u_{\xd,\varepsilon}}_{W^{2s-\xg,p}(D)}\leq C \| \nu \|_{\mathfrak{M}(\Omega\setminus\{0\};d(x)^s|x|^{\tau_+})}, \quad \forall \delta>0,\, \varepsilon>0.
\ea
 Since $u_{\delta,\varepsilon} \to u_{\delta}$ weakly in $W^{2s-\gamma}(D)$, it follows that
\ba\label{ineq4}
\norm{u_{\xd}}_{W^{2s-\xg,p}(D)}\leq \liminf_{\varepsilon \to 0}\norm{u_{\xd,\varepsilon}}_{W^{2s-\xg,p}(D)} \leq C \| \nu \|_{\mathfrak{M}(\Omega\setminus\{0\};d(x)^s|x|^{\tau_+})}, \quad \forall \delta>0.
\ea 
This implies that there exists a subsequence, denoted by the same index $\xd,$ and $u\in L^1_{\loc}(\xO\setminus\{0\})$ such that $u_\xd\to u$ a.e. in $\xO.$ Furthermore, by \eqref{u>k-1}, we can easily show that for any $r>0$ and $q\in[1,2_s^*),$ $u_\xd\to u$  in $L^q(\xO\setminus B_r(0)).$

Let $\xs>0$ and $b<a<2s-\tau_+$. Then by \eqref{wheightineq}, for any $\delta > 0$, we have
\ba\label{wheightineq2}
\int_{B_{\xe}(0)}|u_\delta||x|^{-b} dx\leq \xe^{a-b}\int_{B_{\xe}(0)}|u_\delta||x|^{-a} dx \leq c\, \xe^{a-b} \int_{\xO\setminus\{0\}} d(x)^s|x|^{\tau_+}d\xn.
 \ea
Hence there exists $\xe_0$ such that for any $\varepsilon \leq \varepsilon_0$, there holds
\ba \label{unx-b}
\int_{B_{\xe}(0)}|u_\delta||x|^{-b} dx\leq \frac{\xs}{4}.
\ea

Since $u_\xd\to u$  in $L^1(\xO\setminus B_\xe(0))$, there exists $\delta_0>0$ such that
\bal
\int_{\xO\setminus B_{\xe}(0)}|u_{\delta}-u||x|^{-b} dx\leq \frac{\xs}{2},\quad\forall \delta \in (0,\delta_0).
\eal
Letting $\delta \to 0$ in \eqref{unx-b} and using Fatou's lemma, we obtain
\bal
\int_{B_{\xe}(0)}|u||x|^{-b} dx\leq \frac{\xs}{4}.
\eal
Therefore, for any $\delta< \delta_0$,
\bal
\int_{\xO\setminus\{0\}}|u_{\delta}-u||x|^{-b} dx\leq\xs,
\eal
which implies that $u_\delta \to u$ in $L^1(\xO;|x|^{-b})$. This, together with the convergence \eqref{nuntonu} and estimate \eqref{estLinfty3}, enables us to pass to the limit in \eqref{solun-1} to obtain
\ba \label{u-sol-1}
\int_\xO u(-\xD)^s_{\tau_+}\psi dx=\int_{\xO\setminus\{0\}}\psi |x|^{\tau_+} d\xn,\quad\forall\psi\in \mathbf{X}_\xm(\xO;|x|^{-b}),
\ea
namely $u$ is a weak solution of \eqref{veryweaksolutionmesure}. Since $u_\delta \geq 0$ a.e. in $\Omega \setminus \{0\}$ for any $\delta > 0$, it follows that $u \geq 0$ a.e. in $\Omega \setminus \{0\}$. Moreover, $u$ is the unique solution to \eqref{veryweaksolutionmesure}.


As for the general case $\xn\geq0,$ we consider a smooth exhaustion of $\xO\setminus\{0\},$ i.e. smooth open sets $\{O_l\}_{l \in\N}$ such that
\bal
O_l \Subset O_{l+1}\Subset\xO\setminus \{0\}\quad\text{and}\quad\cup_{l \in \N} O_l=\xO\setminus\{0\}.
\eal
Set $\xn_l = \1_{\overline{O}_l}\xn$ and let $u_l \in L^1(\xO;|x|^{-b})$ be the nonnegative weak solution to \eqref{veryweaksolutionmesure}, namely
\ba \label{ulsol}
\int_\xO u_l(-\xD)^s_{\tau_+}\psi dx=\int_{\xO\setminus\{0\}}\psi |x|^{\tau_+} d\xn_\ell,\quad\forall\psi\in \mathbf{X}_\xm(\xO;|x|^{-b}).
\ea
For $l>l'$, since $\nu_l - \nu_{l'} \geq 0$, it follows that $u_l \geq u_{l'}$ a.e. in $\Omega \setminus \{0\}$. We have
\ba \label{ull'}
\int_\xO (u_l-u_{l'})(-\xD)^s_{\tau_+}\psi dx=\int_{\xO\setminus\{0\}}\psi |x|^{\tau_+} d(\nu_l - \nu_{l'}),\quad\forall\psi\in \mathbf{X}_\xm(\xO;|x|^{-b}).
\ea

By taking $\psi=\xi_b$ (where $\xi_b$ is the solution to \eqref{xib}) in \eqref{ull'}, we deduce
\bal
\| u_l - u_{l'} \|_{L^1(\Omega;|x|^{-b})} &=\int_{\Omega}(u_l - u_{l'})|x|^{-b}dx   \\
&= \int_{\xO\setminus\{0\}}\xi_b |x|^{\tau_+} d(\nu_l - \nu_{l'}) \leq c\,\| \nu_{l} - \nu_{l'} \|_{\GTM(\Omega \setminus \{0\};d(x)^s|x|^{\tau_+})},
\eal
where in the last inequality we have used \eqref{estLinfty3}. Since $\nu_{l} \to \nu$ strongly in $\GTM(\Omega \setminus \{0\};|x|^{\tau_+})$ as $l \to \infty$, from the above estimates,
we see that $\{u_l \}_{l \in \N}$ is a Cauchy sequence in  $L^1(\xO;|x|^{-b}),$ which implies that there exists a function $u$ such that $u_l \to u$ in $L^1(\Omega;|x|^{-b})$. Letting $l \to \infty$ in \eqref{ulsol} yields \eqref{u-sol-1}. Therefore $u$ is a weak solution of \eqref{veryweaksolution}.

\medskip

\noindent \textit{A priori estimate.}  Taking $\psi =\xi_b$ (where $\xi_b$ is the solution to \eqref{xib}) as a test function in \eqref{weakform-a1}, we obtain \eqref{l1ineqb}. \medskip

\noindent \textit{Monotonicity.}  Assume that  $\nu_1, \nu_2 \in \GTM(\Omega \setminus \{0\};|x|^{\tau_+})$ such that $\nu_1 \leq \nu_2$ and let $u_i$ be the unique weak solution to \eqref{veryweaksolution} with $\nu$ replaced by $\nu_i$, $i=1,2$. From the above construction, we see that $u_1 \leq u_2$ a.e. in $\Omega \setminus \{0\}$. Consequently, we see that if $\nu \geq 0$ then the weak solution $u$ to \eqref{veryweaksolutionmesure} satisfies $u \geq 0$ a.e. in $\Omega \setminus \{0\}$. \medskip

\noindent \textit{Uniqueness.} The uniqueness is a consequence of the monotonicity. 
\end{proof}

Next we treat the case where the source is a measure on the whole domain $\Omega$. Recall that function $\Phi_{s,\mu}^{\Omega}$ satisfies \eqref{PhiOmega}. In view of the proof of Lemmas \ref{smoothlemma} and \ref{smoothlemma2}, $\xF_\xm^\xO\in C^{2s+\xb_0}(\xO\setminus\{0\})$ for some $\xb_0>0$ and for any $\xe>0$ such that $B_{4\xe}(0)\subset\xO$ there holds
\bal
\norm{d^{-s}\Phi_{s,\mu}^{\Omega}}_{C^\xa(\overline{\xO\setminus B_\xe(0)})}+\norm{\Phi_{s,\mu}^{\Omega}}_{C^s(\overline{\xO\setminus B_{\xe}(0)})}
\leq c,
\eal
where $c>0$ depends only on $\xO,s,\xm,\xe$.

\begin{proof}[\textbf{Proof of Theorem \ref{existence3-dirac}}]
By virtue of Theorem \ref{inter}, there exists a unique weak solution $u_\nu \in L^1(\Omega;|x|^{-b})$ of problem \eqref{veryweaksolution}, namely
\ba \label{unu-sol}
\int_\xO u_{\nu}(-\xD)^s_{\tau_+}\psi dx=\int_{\xO\setminus\{0\}}\psi |x|^{\tau_+} d\xn,\quad\forall\psi\in \mathbf{X}_\xm(\xO;|x|^{-b}).
\ea	
Put $u_{\nu,\ell}=u_\nu + \ell \Phi_{s,\mu}^{\Omega}$. Take $\psi \in \mathbf{X}_{\mu}(\Omega;|x|^{-b})$, then $|(-\Delta)_{\tau_+}^s \psi(x)| \leq C_{\psi}|x|^{-b}$ for a.e. $x \in \Omega \setminus \{0\}$. By noticing that $\tau_- - b > -N$, we deduce that
\bal
\left|\int_{\Omega} \Phi_{s,\mu}^{\Omega}(x) (-\Delta)_{\tau_+}^s \psi (x) \right| \leq C_\psi \int_{\Omega} |x|^{\tau_- - b}dx < + \infty.
\eal
Therefore, from \eqref{unu-sol}, we obtain \eqref{weakform-a1}.

The uniqueness of the weak solution to \eqref{veryweaksolution} follows from Theorem \ref{existence2}.
\end{proof}



\end{document}